%% file: main.tex
\documentclass[10pt]{article}

\usepackage[utf8]{inputenc}							%
\usepackage[T1]{fontenc}						%
\usepackage[normalem]{ulem}
\usepackage[a4paper]{geometry}
\usepackage[font=footnotesize]{caption}
\usepackage{amsfonts}
\usepackage{amsmath}
\usepackage{amssymb}
\usepackage{amsthm}
\usepackage{mathtools}
\mathtoolsset{
  showonlyrefs=false
}
\usepackage{url}

\usepackage[dvipsnames,svgnames]{xcolor}						%
\colorlet{MyBlue}{DodgerBlue!75!Black}
\colorlet{MyGreen}{DarkGreen!85!Black}
\colorlet{MyGray}{White!75!Black}
\usepackage{hyperref}
\hypersetup{
final,
colorlinks=true,
linktocpage=true,
pdfstartview=FitH,
breaklinks=true,
pdfpagemode=UseNone,
pageanchor=true,
pdfpagemode=UseOutlines,
plainpages=false,
bookmarksnumbered,
bookmarksopen=false,
bookmarksopenlevel=1,
hypertexnames=true,
pdfhighlight=/O,
urlcolor=Maroon,linkcolor=MyBlue!60!black,citecolor=DarkGreen!70!black,	%
pdftitle={},
pdfauthor={},
pdfsubject={},
pdfkeywords={},
pdfcreator={pdfLaTeX},
pdfproducer={LaTeX with hyperref}
}

\theoremstyle{plain}
\newtheorem{theorem}{Theorem}[section]
\newtheorem{lemma}[theorem]{Lemma}
\newtheorem{corollary}[theorem]{Corollary}
\newtheorem{proposition}[theorem]{Proposition}

\theoremstyle{definition}
\newtheorem{definition}[theorem]{Definition}
\newtheorem{example}[theorem]{Example}
\theoremstyle{remark}
\newtheorem{remark}[theorem]{Remark}

\usepackage{environ}
\NewEnviron{align+}{%
\begin{align}
  \BODY
\end{align}
}
\NewEnviron{equation+}{%
\begin{equation}
  \BODY
\end{equation}
}

\title{Hybrid Methods in Polynomial Optimisation\thanks{The authors are ordered alphabetically.}}

\author{
  Johannes Aspman\thanks{Faculty of Electrical Engineering, Czech Technical University of Prague, the Czech Republic.}
  \and Gilles Bareilles\footnotemark[2]
  \and Vyacheslav Kungurtsev\footnotemark[2]
  \and Jakub Mare\v{c}ek\footnotemark[2]
  \and Martin Tak\'a\v{c}\thanks{Mohamed bin Zayed University of Artificial Intelligence, Masdar City, Abu Dhabi, United Arab Emirates.}
}

\usepackage{xspace}
\newcommand{\ie}{\emph{i.e.},\xspace}
\newcommand{\eg}{\emph{e.g.},\xspace}

\usepackage[export]{adjustbox}
\usepackage[caption=false]{subfig} %

\usepackage{cleveref}

\usepackage{algorithm}
\usepackage{algpseudocodex}

\newcommand{\debug}[1]{#1}                                  %

\newcommand{\newmacro}[2]{\newcommand{#1}{\debug{#2}}}		%
\newcommand{\newop}[2]{\DeclareMathOperator{#1}{\debug{#2}}}		%

\usepackage{color-edits}
\addauthor{gb}{DodgerBlue}
\addauthor{JM}{orange}
\addauthor{JA}{red}

\newmacro{\ab}{$\alpha$-$\beta$}

\newmacro{\R}{\mathbb{R}}
\newmacro{\N}{\mathbb{N}}
\newmacro{\RR}{\mathbb{R}}
\newmacro{\bbR}{\mathbb{R}}
\newmacro{\bbN}{\mathbb{N}}

\newmacro{\E}{\mathcal{E}}
\newmacro{\I}{\mathcal{I}}
\newmacro{\mE}{m_\mathcal{E}}
\newmacro{\mI}{m_\mathcal{I}}

\newmacro{\C}{\mathbb{C}}

\newmacro{\xtilde}{\tilde x}
\newmacro{\x}{x}
\newmacro{\y}{y}
\newmacro{\stepsize}{\eta}
\newmacro{\st}{\text{s.t.}}

\newmacro{\mfR}{\Re}
\newmacro{\mfI}{\Im}
\newmacro{\RK}{R^{(k)}}

\newop{\trace}{trace}
\newop{\tr}{tr}
\newop{\dist}{dist}

\newmacro{\eqdef}{\;\stackrel{\mathclap{\normalfont\tiny\mbox{def}}}{=}\;}

\newmacro{\nhd}{\mathcal{N}}
\newmacro{\ball}{\mathcal{B}}

\newcommand{\vx}{\debug{x}}                  %
\newcommand{\vPsys}{\debug{z}}               %
\newcommand{\opt}[1][\vx]{{#1}^{\debug{\star}}}  %
\newcommand{\adh}[1][\vx]{\debug{\hat{#1}}}  %

\newcommand{\curr}[1][\vx]{{#1}_{\debug{k}}}  %

\newmacro{\Lag}{\mathcal{L}}  %

\newmacro{\idfun}{r}
\newmacro{\Actset}{\mathcal{A}}
\newmacro{\Psys}{F}
\newmacro{\Psysn}{p}
\newmacro{\Psysm}{q}

\newmacro{\Sol}{\mathcal{S}}
\newmacro{\SolDual}{\Sol_{\mathcal{D}}}

\newmacro{\dualSet}{\bbR^{\mE}\times\bbR^{\mI}_{+}}

\newmacro{\bigoh}{\mathcal{O}}

\newmacro{\D}{\operatorname{D}}     %
\newmacro{\Jac}{\operatorname{Jac}} %
\newmacro{\diag}{\operatorname{diag}}
\newmacro{\rank}{\operatorname{rank}}

\newmacro{\degree}{\operatorname{deg}}
\newmacro{\parts}{\mathcal{P}}

\date{}

\begin{document}

\maketitle

\input{body}

\end{document}

%% file: body.tex
\begin{abstract}
  The Moment/Sum-of-squares hierarchy provides a way to compute the global minimizers of polynomial optimization problems (POP), at the cost of solving a sequence of increasingly large semidefinite programs (SDPs).
  We consider large-scale POPs, for which interior-point methods are no longer able to solve the resulting SDPs.
  We propose an algorithm that combines a first-order method for solving the SDP relaxation, and a second-order method on a non-convex problem obtained from the POP.
  The switch from the first to the second-order method is based on a quantitative criterion, whose satisfaction ensures that Newton's method converges quadratically from its first iteration.
  This criterion leverages the point-estimation theory of Smale and the active-set identification.
  We illustrate the methodology to obtain global minimizers of large-scale optimal power flow problems.
\end{abstract}



\section{Introduction}

\subsection{Context: global optimization of polynomial problems}
\label{sec:cont-glob-optim}

In this paper, we consider Polynomial Optimization Problems (POP, \cite{lasserre2001global,lasserre2015introduction}) of the form
\begin{equation}\label{eq:POP}
  \begin{aligned}
    \opt[f] =& \min_{x\in\mathbb{R}^n} & & f(x) \\
           & \text{s.t.}  & & g_i(x) \geq 0, \quad i \in \I \\
           &&& g_j(x) = 0, \quad j \in \E
  \end{aligned}
\end{equation}
where $f, g_i, g_j:\R^n \to \R$ are multivariate polynomials, and $\I$, $\E$ denote index sets.
Such problems appear in a variety of applications, ranging from power engineering \cite{ghaddarOptimalPowerFlow2016} to mechanics \cite{tyburec2021global} to optimal control.

Global optimization of polynomial problems can be tackled with the Moment / Sum-of-Squares (SoS) methodology \cite{lasserre2001global,Handbook}.
It allows formulating a hierarchy of increasingly large semidefinite relaxations.
The key feature of this methodology is that the minimum value of the convex relaxations converge to the \emph{global} minimum of the polynomial problem, in spite of its non-convex nature and the existence of many local minima.

At each step of the Moment/SoS hierarchy, we need to solve a semidefinite program (SDP).
While there are exact methods \cite{henrionExactAlgorithmsSemidefinite2021}, interior-point methods are the most commonly used for solving the SDP relaxations.
The main strength of interior-point methods is their fast quadratic convergence rate, which allows to obtain arbitrary precision minimizers in a small number of iterations independent of the problem dimension \cite{nesterov1994interior}.
However, the memory and computational cost of one iteration grows quickly with the problem dimension.
Even though this issue can be attenuated by leveraging sparsity in problem \eqref{eq:POP} \cite{magronSparsePolynomialOptimization2023,wangCSTSSOSCorrelativeTerm2021,joszLasserreHierarchyLarge2018}, the second-order relaxation of large-scale industrial instances of \eqref{eq:POP} is out of reach for current SDP solvers.
Several studies have also noted that classical interior-point solvers encounter numerical difficulties or even provide poor quality solutions when applied to relaxations of polynomial problems \cite{wakiStrangeBehaviorsInteriorpoint2012,wangCertifyingGlobalOptimality2022}.
As a result, much research has focused on designing alternative methods to solve SDPs coming from the Moment/SoS hierarchy.

A recent trend is the use of first-order methods for solving the SDP relaxations.
Their cheap iteration cost and their ability to leverage the low-rank structure of solutions make them applicable to large-scale SDPs.
However, they can only produce low-accuracy solutions, due to their sublinear convergence rates.
The popular Burer-Monteiro approach \cite{burerNonlinearProgrammingAlgorithm2003} reformulates the SDP into a non-convex problem, with guarantees on the recovery \cite{boumalDeterministicGuaranteesBurerMonteiro2020,waldspurgerRankOptimalityBurer2020}.
On these non-convex reformulations, coordinate-descent schemes are particularly efficient in finding a low-accuracy minimizer; see \eg{} \cite{marecekLowrankCoordinatedescentAlgorithm2017} for an application in power systems.
Splitting techniques, originating in nonsmooth convex optimization, also use the Burer-Monteiro approach to exploit the low-rank structure of solutions; see \eg{} \cite{o2016conic,souto2020exploiting,Garstka_2021}.
Besides, semidefinite programs with a constant trace property are common in applications, including in Moment/SoS relaxations \cite{maiExploitingConstantTrace2022}; specific first-order methods include stochastic algorithms \cite{nesterovRandomizedMinimizationEigenvalue2023,daspremontStochasticSmoothingAlgorithm2014}, the augmented lagrangian of \cite{yurtseverScalableSemidefiniteProgramming2021}, and spectral bundle methods that may incorporate second-order information \cite{helmbergSpectralBundleMethod2000,helmbergSpectralBundleMethod2014,nollSpectralBundleMethods2005,haaralaGloballyConvergentLimited2007}.
Finally, one may employ (inexact and accelerated) projected gradient methods \cite{yangInexactProjectedGradient2022}.
The complexity of the above-mentioned first-order methods is not well understood, in practice they attain relative suboptimality accuracies of $10\%$ to $1\%$; see \eg{} \cite{yurtseverScalableSemidefiniteProgramming2021,maiExploitingConstantTrace2022}.
A notable exception is the stochastic algorithms \cite{nesterovRandomizedMinimizationEigenvalue2023,daspremontStochasticSmoothingAlgorithm2014} which come with nonasymptotic guarantees; \eg{} \cite[Alg. 2.1]{nesterovRandomizedMinimizationEigenvalue2023} attains a relative suboptimality of $\delta$ in $\bigoh(\delta^{-2})$ iterations.
This slow convergence seem to be inherent in first-order methods.

Independently, it is possible to compute local minimizers of POP to high accuracy using Newton's method,
which underlies the so-called homotopy-continuation approach \cite{alexander1978homotopy,allgower2003introduction,bates2013numerically}.
Indeed, Newton's method converges to solutions at a fast quadratic --or doubly exponential-- speed, which allows to obtain solutions with arbitrary precision in a small number of iterations, independent of conditioning or dimension.
However, \emph{(i)} in order to benefit from the fast quadratic rate, the method must be started close enough to a minimizer, and \emph{(ii)} it only works on smooth problems: the POP is nonsmooth, because of the inequality constraints.

In this paper, we propose an algorithm that \underline{\it provably converges to the global minimizer} of \eqref{eq:POP} with \underline{\it local quadratic speed}, as discussed below.
The algorithm combines a first-order method on the convex relaxation and Newton's method on the polynomial problem.
\Cref{fig:algoscheme} shows the typical behavior of the proposed algorithm on the small but challenging WB2 instance of Alternating Current Optimal Power Flow \cite{bukhshLocalSolutionsOptimal2013}, a power system polynomial problem detailed in \cref{sec:experiments}.

\begin{figure}[h]
  \centering
    \includegraphics[width=0.76\textwidth]{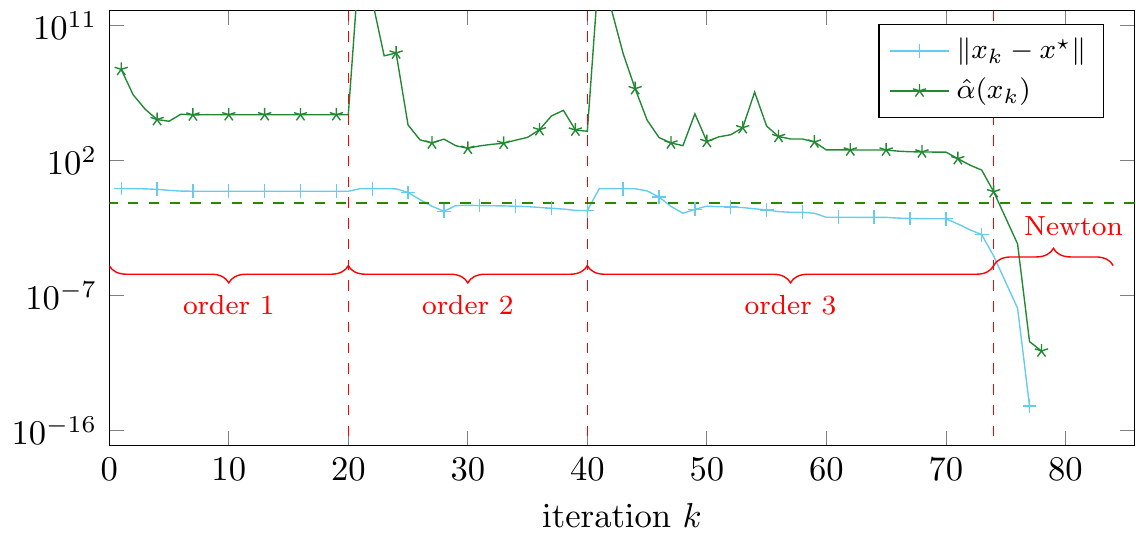}
 \caption{
   Typical behavior of the proposed \cref{alg:hybrid} on the WB2 instance of ACOPF (see \cref{sec:experiments}): the method partially solves the SDP relaxations of order $1$, $2$, and $3$ until the $\alpha$-test is satisfied ($\hat{\alpha}(x_{74}) \le 0.15$), at which point it switches to Newton's method.
   The SDP relaxations are solved with the Mosek.
   Newton's method \emph{converges quadratically right from the first iteration}.
   This is a major distinction with traditional two-phases methods: there, the switch from the two phases is usually heuristic, so that the Newton-based second phase is guaranteed to converge at quadratic rate only after a finite \emph{but unknown} number of iterations; see \cref{fig:toypb_slowNewton}.
    \label{fig:algoscheme}
 }
\end{figure}

\begin{example}[Two-dimensional example]%
  \label{ex:toypb}
  We introduce a two-dimensional pop that will support illustrations along the development, defined as follows:
  \begin{equation}%
    \label{eq:toypbfuns}
    \begin{aligned}
      & \min_{x\in\bbR^{2}}
      & & f(x) = -2.5x_{1}^2 + 3x_{1}x_{2} - 2.5x_{2}^2 - 3x_{1} + 5x_{2} - 2.5 \\
      & \text{s.t.} &&
                       g_{1}(x) = -0.5x_{1}^3 + x_{2} \ge 0 \\
      &&& g_{2}(x) = -0.05x_{1}^2 - x_{2} + 1.8 \ge 0 \\
      &&& g_{3}(x) = -0.05x_{2}^2 + x_{1} + 0.1x_{2} + 0.35 \ge 0
    \end{aligned}
  \end{equation}
  \Cref{fig:toypb_levels} shows the level lines of the objective and the boundaries of the constraint set.
  The unique minimizer is $\opt \approx (0.83271, -0.71130)$, with function value $f(\opt) \approx -4.77529$.
\end{example}

\begin{figure}[h]
  \centering
  \subfloat[]{
    \label{fig:toypb_levels}
    \includegraphics[width=0.46\textwidth]{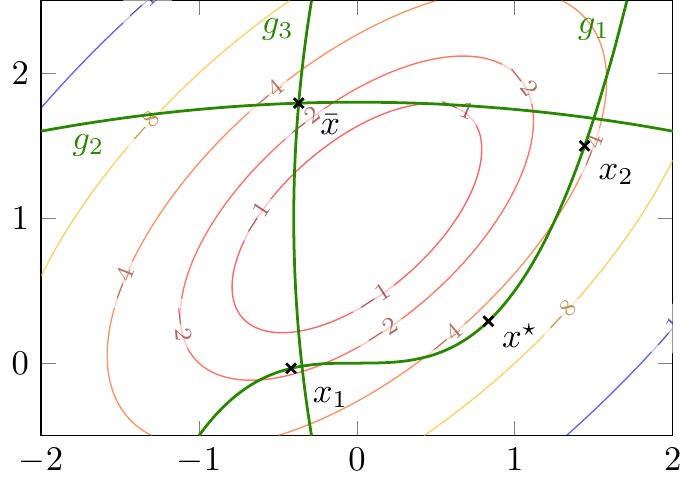}
  }
  \subfloat[]{
    \label{fig:toypb_slowNewton}
    \includegraphics[width=0.46\textwidth]{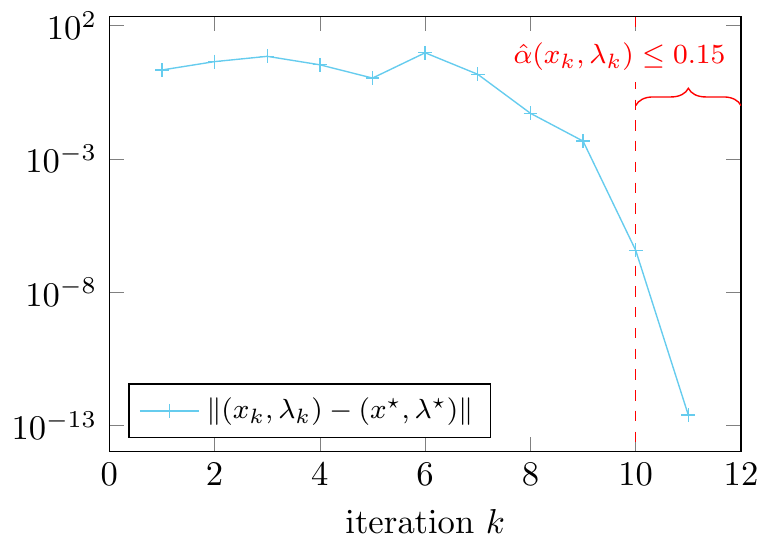}
  }
 \caption{
   The left pane displays problem \eqref{eq:toypbfuns}: the three green lines show the points which cancel each constraint polynomial, and the level curves represent the objective.
   The global minimum is $\opt$, $\bar{x}$ is a local minimum.
   The active constraint at point $\opt$ is $g_{1}$.
   Furthermore, points $x_{1}$ and $x_{2}$ are stationary points for the reduced problem $\min f$ s.t. $g_{1} = 0$, revealing that the second-order method may be attracted by other points than the global minimizer.
   The right-pane displays the distance between the iterates of Newton's method $(\curr, \curr[\lambda])$ and the optimal primal dual pair $(\opt, \opt[\lambda])$, when applied to the KKT conditions of the reduced problem $\min f$ s.t. $g_{1} = 0$.
   Newton's method fails to produce a significant improvement during the first $9$ iterations, and only then enters the fast convergence regime.
   This illustrates that the quadratic rate of Newton's method may only occur after a (possibly large) number of iterations, hence the interest of the \ab{} criterion.
    \label{fig:toypb}
 }
\end{figure}

\subsection{Active constraints identification and Alpha-Beta theory}
\label{sec:alpha-beta-theory}

First-order methods applied to convex relaxations can provide efficiently a low-accuracy approximation of the global minimizer of \eqref{eq:POP}.
Newton's method can be used to refine its accuracy, with the following tools.

\paragraph{Active constraint identification}
Minimizers of \eqref{eq:POP} are characterized by a system of nonlinear equations and inequations, the so-called KKT conditions.
However, Newton's method solves systems of smooth equations only, and thus cannot be applied directly.
At a minimizer, the inequality constraints split in the \emph{active constraints}, which are exactly null, and the \emph{inactive constraints}, which have a positive value.%
If the inequalities active at a minimizer are known, then, near this minimizer, the problem simplifies: the inactive constraints can be discarded while the active constraints set to zero without changing the solution.
On \cref{ex:toypb}, illustrated in \cref{fig:toypb_levels}, one can replace near $\opt$ the constraints $g_{1}, g_{2}, g_{3}\ge0$ by $g_{1}=0$ while preserving $\opt$ as a solution.
The obtained \emph{reduced problem} thus features only equalities: its minimizers are characterized by a system of equations, which can now be solved by Newton's method.
In order to make this approach practical, the set of active constraints should be obtained without knowledge of the minimizer.
Under some qualifying conditions, it can be detected from  any point in a neighborhood of the minimizer; see \eg{} \cite{oberlinActiveSetIdentification2006} for a generic take.

\paragraph{\ab{} theory} The point-estimation or \ab{} theory provides a computable criterion that ensures that, at a given point, \emph{i)} there exists a nearby zero of the systems of equations at hand, and \emph{ii)} that Newton's method converges to that zero with a quadratic rate from the first iteration \cite{Cucker1999,ShubSmale1993}.
This criterion can be seen as a quantitative version of classical results, which guarantee only eventual quadratic convergence of Newton's method, under some difficult-to-check conditions; see \eg{} \cite[Th. 13.6]{bonnansNumericalOptimizationTheoretical2006} and references therein.
In particular, away from minimizers, Newton's method for unconstrained minimization can be as slow as gradient descent \cite{cartisComplexitySteepestDescent2010}; see \cref{fig:toypb_slowNewton} for an illustration.

\subsection{Contributions and outline}
\label{sec:contributions}

The purpose of this paper is two-fold.
First, we give an overview and results of the fields of convex relaxations of polynomial problems, activity identification, and point-estimation theory.
We believe that nice interactions are possible between these different communities.
Second, we combine these methodologies in a hybrid algorithm for global optimization of polynomials, with quadratic local convergence.
Notably, the \ab{} theory provides a \emph{theoretically grounded} practical test to perform the switch from the first-order to the second-order method, where traditional two-phase methods rely on heuristic switch conditions.
The proposed hybrid algorithm enjoys the following guarantees; see \cref{th:algocv} for a precise statement.
\begin{theorem}[Main result, Informal version]
  Consider a POP \eqref{eq:POP} that admits a unique global minimizer $\opt$.
  Then, the iterates of the forthcoming \cref{alg:hybrid} converge to the global minimizer.

  Under classical assumptions on $\opt$, \cref{alg:hybrid} eventually detects the correct active set, switches to Newton's method, and converges quadratically.
\end{theorem}

The outline of the remainder of the paper is as follows.
First, in \Cref{sec:momentsoshierarchy}, we introduce the hierarchy of moment relaxations and recall the main convergence result.
In \Cref{sec:active-set-estim}, we develop a procedure specific to polynomial optimizations problems for identifying the constraints active at a minimizer from neighbouring points.
While such techniques are well-known for non-linear programs involving $\mathcal{C}^{2}$-smooth functions, we show identification under mild assumptions, by leveraging the algebraic nature of the problem.
In \Cref{sec:algebr-geom-point}, we review the main results of the \ab{} theory and present the computable criterion that guarantees the fast convergence of Newton's method.
In \Cref{sec:hybr-meth-polyn}, we detail the proposed algorithm combining first-order method on the SDP and second-order method on the POP, and give convergence guarantees.
Finally, we present in \Cref{sec:experiments} numerical illustration of the hybrid method.

\paragraph{Notations}%
\label{sec:notation}

We let $x_{+} \eqdef \max(x, 0)$ denote the (pointwise) positive part of $x$, and $\|\cdot\|$ denotes the standard euclidean norm.

A degree $d$ polynomial of variable $x\in\bbR^{\Psysn}$ is written as
  $g(x) = \sum_{|\nu |\leq d} a_{\nu} x^{\nu}$,
where $\nu = (\nu_{1}, \ldots, \nu_{\Psysn})\in\bbN^{\Psysn}$ denote exponents, $x^{\nu} = x_1^{\nu_1} x_2^{\nu_2} \cdots x_{\Psysn}^{\nu_{\Psysn}}$ denote monomials, and we let $|\nu| \eqdef \sum_{i=1}^{\Psysn} \nu_i$.
We use the following polynomial norm
\begin{equation}
  \|g\|_{p}^{2} \eqdef \sum_{|\nu|\leq d} |a_\nu|^2 \frac{\nu_{1}! \ldots \nu_{n}! (d-|\nu|)!}{d!}.
\end{equation}
In turn, the norm of a system of polynomial equations $\Psys:\bbR^{\Psysn} \to \bbR^{\Psysm}$ defines as $\|\Psys\|_{p} \eqdef \left(\sum_i\|\Psys_i\|_{p}^2 \right)^{1/2}$.
Finally, we let $\|x\|_1^{2} \eqdef 1+\sum_{i=1}^n x_i^2$.

\section{Global optimization with convex relaxation hierarchies}%
\label{sec:momentsoshierarchy}

In this section, we recall the hierarchy of moment relaxations and state known results.
We follow the notations of the book \cite{lasserre2015introduction}, and refer to it for an in-depth treatment of this topic.

We first introduce relevant objects.
We let $\bbR[X]$ denote the set of polynomials in variable $X = (X_{1}, \ldots, X_{n})$ with real coefficients, and $\bbN_{d}^{n}$ denote the set of $n$-tuples $(\alpha_{1}, \ldots, \alpha_{n})$ whose elements sum to $d$ at most.
With $y = (y_{\alpha})_{\alpha\in\N^{n}}$ a sequence indexed by exponents in $\N^{n}$, $L_{y}:\bbR[X] \to \bbR$ denotes the \emph{Riesz} linear functional $f=\sum_{\alpha} f_{\alpha}X^{\alpha} \mapsto \sum_{\alpha}f_{\alpha}y_{\alpha}$.
The $d$-th order \emph{moment matrix} $M_{d}(y)$ associated with $y=(y_{\alpha})_{\alpha}$ is defined by $M_{d}(y)[\alpha, \beta] \eqdef L_{y}(X^{\alpha} X^{\beta}) = y_{\alpha+\beta}$, for all $\alpha, \beta\in\N_{d}^{n}$, and the $d$-th order \emph{localizing matrix} $M_{d}(gy)$ associated with $y$ and $g = \sum_{\alpha}g_{\alpha}X^{\alpha}$ is defined by $M_{d}(gy)[\alpha, \beta] \eqdef L_{y}(g  X^{\alpha} X^{\beta}) = \sum_{\delta} g_{\delta}y_{\alpha+\beta+\delta}$, for all $\alpha, \beta\in\N_{d}^{n}$.
Finally, we let $r_{j} \eqdef \lceil (\deg g_{j})/2 \rceil$ for $j\in\I\cup\E$.

The \emph{order-$r$ moment relaxation} of \eqref{eq:POP}, for $r \ge r_{0} \eqdef \max( \lceil (\deg f)/2 \rceil, r_{j} )$, defines as the following SDP:
\begin{equation}%
  \label{eq:sdprel}
  \begin{aligned}
    \rho_{r} = & \inf_{y \in \bbR^{\bbN^{n}_{2r}}}
    & & L_{y}(f) \\
            & \text{s.t.}
    & & M_{r}(y) \succeq 0 \\
            &&& M_{r - r_{j}}(g_jy) = 0, \quad j \in \E \\
            &&& M_{r - r_{j}}(g_jy) \succeq 0, \quad j \in \I \\
            &&& y_{0} = 1
  \end{aligned}
\end{equation}
This is a semidefinite program in dual form, which optimizes over sequences indexed by exponents of degree up to $2r$.

We can now state the main convergence result for the sequence of moment relaxations \eqref{eq:sdprel}: the optimal values $\rho_{r}$ converge to $f^{\star}$ from below, and the first degree coefficients of a solution of \eqref{eq:sdprel} converge to the global minimizer as $r \to \infty$, as long as the global minimizer is unique.
\begin{proposition}[Convergence of the moment sequence]\label{prop:momentrel}
  Consider a POP, and the semidefinite relaxations \eqref{eq:sdprel}.
  Then,
  \begin{enumerate}
    \item $\rho_{r} \uparrow f^{\star}$ as $r\to\infty$,
    \item Let $y_{r}$ be a nearly optimal solution of \eqref{eq:sdprel}, with \eg $L_{y_{r}}(f) \le \rho_{r} + \frac{1}{r}$. \label{prop:momentrelii}
      If $P$ has a unique minimizer $\opt$, then $L_{y_{r}}(X_{j}) \to \opt_{j}$ as $r\to\infty$.
  \end{enumerate}
\end{proposition}
\begin{proof}
  This result is nicely summarized by \cite[Th. 6.2]{lasserre2015introduction} and \cite[Th. 12 \& Cor. 13]{schweighofer2005optimization},
  based on the original work of Lasserre \cite{lasserre2001global}.
\end{proof}

\begin{example}[Moment hierarchy]%
  \label{ex:toypb-moment}
  We apply the moment hierarchy \eqref{eq:sdprel} on the problem introduced in \cref{ex:toypb}.
  \Cref{tab:toypbhierarchy} shows for the order $r=2$ and $3$ the optimal value, the extracted point $(L_{y_{r}}(X_{1}), L_{y_{r}}(X_{2}))$, and the measure extracted from the solution $y_{r}$ by the procedure in \cite{henrion2005detecting}.
  \Cref{tab:toypbextractdata} details the extracted measures.
  The correct function value and minimizer are obtained at the third order of the relaxation.
  \Cref{fig:toypb_SDPiterates_a} shows the trajectory $(L_{y^{k}_{3}}(x_{1}), L_{y^{k}_{3}}(x_{2}))$ in the pop space, where $(y_{3}^{k})_{k}$ denotes the iterates of a first-order method (here the Primal Dual Hybrid Gradient algorithm \cite{souto2020exploiting}) applied to \eqref{eq:sdprel} with $r=3$.
  \Cref{fig:toypb_SDPiterates_b} shows the distance between the iterate $(L_{y^{k}_{3}}(X_{1}), L_{y^{k}_{3}}(X_{2}))$ in the pop space and the global minimizer.
\end{example}

\begin{table}[tbp]
  \footnotesize
  \caption{
    Order two and three relaxations of problem~\eqref{eq:toypbfuns}: optimal value $\rho_{r}$, simple minimizer extraction as described in \cref{prop:momentrel}ii), and refined extraction yielding a combination of Dirac measures \cite{henrion2005detecting}; see \cref{rmk:minextraction}.
    The third order relaxation is exact.
    \Cref{tab:toypbextractdata} describes points $y_{1}$, $y_{2}$, and $y_{3}$.
    \label{tab:toypbhierarchy}
  }
  \begin{center}
      \begin{tabular}{lccc}
        Relaxation order $r$ & $\rho_r$ & $(L_{y_{r}}(X_{1}), L_{y_{r}}(X_{2}))$ & refined extraction \cite{henrion2005detecting} \\ \hline
        $2$ & $-29.34644$ & $(0.017, -0.563)$ & $0.18 \delta_{y_{1}} + 0.82 \delta_{y_{2}}$\\
        $3$ & $-4.77529$ & $y_3$ & $\delta_{y_{3}}$
      \end{tabular}
  \end{center}
\end{table}

\begin{table}[tbp]
  \footnotesize
  \caption{Points extracted from relaxations of problem \eqref{eq:toypbfuns} with the procedure outlined in \cite{henrion2005detecting}.
    Note that for each point, some of the constraints are active (up to the solver accuracy of about $10^{-9}$).
    They are denoted by $\approx 0$.
  }\label{tab:toypbextractdata}
  \begin{center}
    \begin{tabular}{ccc|ccc}
      point $p$ & point value & f(p) & \multicolumn{3}{c}{$g(y_r)$} \\ \hline
      $y_{1}$ & $(1.69926, -5.47959)$ & $-145$ & $-7.9$ & $7.1$ & $\approx 0$ \\
      $y_{2}$ & $(-0.36854, 1.79321)$ & $-2.8$ & $1.8$ & $\approx 0$ & $\approx 0$ \\
      $y_{3}$ & $(0.83271,  0.28870)$ & $-4.77529$ & $\approx 0$ & $1.5$ & $1.2$
    \end{tabular}
  \end{center}
\end{table}

\begin{figure}[tbp]
  \centering
  \subfloat[]{\label{fig:toypb_SDPiterates_a}\includegraphics[width=0.48\textwidth]{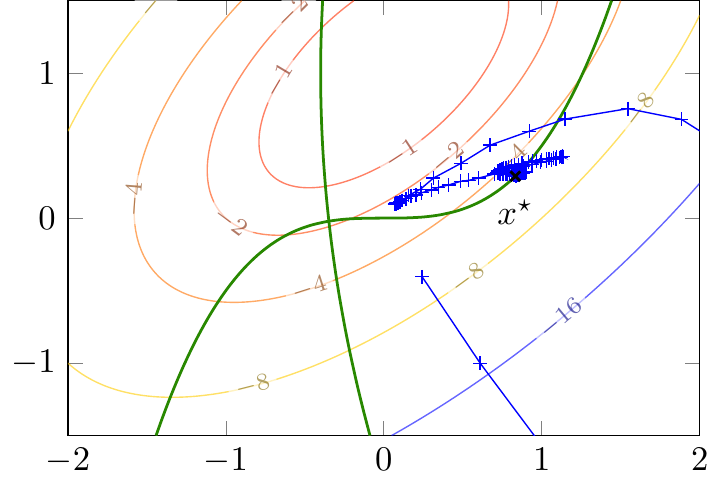}}
  \subfloat[]{\label{fig:toypb_SDPiterates_b}\includegraphics[width=0.48\textwidth]{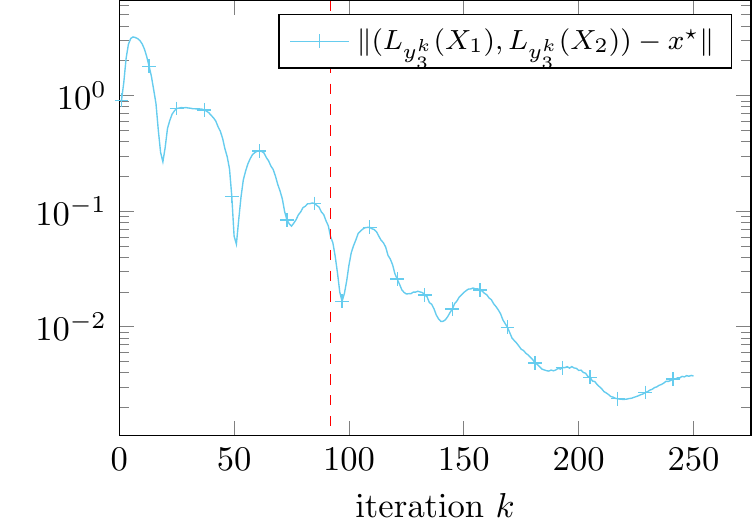}}
  \caption{%
    Behaviour of the iterates generated by a first-order SDP solver on the order-$3$ moment relaxation of \eqref{eq:toypbfuns}.
    The left pane represents the points $(L_{y_{3}^{k}}(X_{1}), L_{y_{3}^{k}}(X_{2}))_{k}$ in the POP space, that correspond to the SDP iterates $(y_{3}^{k})_{k}$; see \cref{prop:momentrel} item ii).
    The right pane show the distance between the extracted points in the POP space and the global minimizer $\opt$ relative to iterations.
    The dashed red line indicates the first time the correct active set is detected.
    }
  \label{fig:toypb_SDPiterates}
\end{figure}

\begin{remark}[Moment relaxation variants]
  We present here the moment relaxation in its simplest form, as first introduced by \cite{lasserre2001global}.
  In an effort to make this approach applicable to large polynomial problems, \cite{ghaddarOptimalPowerFlow2016,wangCSTSSOSCorrelativeTerm2021} among others propose variants that use forms of sparsity in \eqref{eq:POP}, \cite{jeyakumarSemidefiniteProgrammingRelaxation2016} modify the hierarchy to handle noncompact feasible sets while still exploiting sparsity, \cite{joszLasserreHierarchyLarge2018} propose a complex moment problem for polynomial problems that are naturally cast in complex variables, and \cite{Navascués_2008} extends the approach to polynomial optimization in noncommutative variables.
\end{remark}

\begin{remark}[Finite convergence of the moment relaxation]
  In practice, one regularly observes that the values $\rho_{r}$ of the relaxation converge finitely to $\opt[f]$: it often happens that with $\bar{r}=1$ or $2$, with $\rho_{\bar{r}} = f^{\star}$.
  Theorem 6.5 in \cite{lasserre2015introduction} gives sufficient conditions for finite convergence to happen: the minimizer $\opt$ should satisfy the linear independence constraint qualification, strict complementarity, and the second-order sufficient condition, and $M - \|x\|^{2}$ should write as $\sum_{i\in\E}\psi_{i}g_{i} + \sum_{i\in\I}\sigma_{i}g_{i}$ for some $M>0$, some polynomials $\psi_{i}$, and some sum-of-square polynomials $\sigma_{i}$.
  As a complementary result, \cite{nieOptimalityConditionsFinite2014} shows that if the feasible set in included in some ball ---the so-called archimedean condition---, then the finite convergence of the moment relaxation happens generically (\ie{} for any input data $f, g$ for \eqref{eq:POP} except on a set of Lebesgue measure zero).
  With a suitable modification, the moment relaxation also enjoys a generic finite convergence property when the feasible set of \eqref{eq:POP} is not compact \cite{jeyakumarPolynomialOptimizationNoncompact2014}.
\end{remark}

\begin{remark}[Extraction of multiple minimizers]\label{rmk:minextraction}
  \cref{prop:momentrel} provides the simplest way to convert approximate minimizers of the moment problems \eqref{eq:sdprel} to points for the POP \eqref{eq:POP} in a convergent way.
  This relies on uniqueness of the global minimizer.
  For polynomial problems \eqref{eq:POP} which admit $p > 1$ distinct minimizers, one may recover the $p$ minimizers from the exact solution of the SDP moment relaxation using the procedure outlined in~\cite{henrion2005detecting}.
  It requires that the exactness of the moment relaxation, and that the moment matrix satisfies a certain rank condition.
  Whether this procedure holds in a more general setting of inexact relaxations, with low-quality solutions of (possibly inexact) relaxations, or when there are an infinite number of minimizers is unclear; see \eg \cite{klepMinimizerExtractionPolynomial2018,laurentSumsSquaresMoment2009}.
  In a similar vein, \cite{henrionAlgebraicCertificatesTruncated2023} focuses on deciding when a finite sequence $(y_{\alpha})_{\alpha}$ is the (truncated) sequence of moments of some measure.
  They provide a way to either recover the measure, or to generate a certificate that no such measure exist.
\end{remark}

\section{Active set estimation for analytic problems}
\label{sec:active-set-estim}

In this section, we propose an identification procedure that leverages the algebraic nature of problem~\eqref{eq:POP}.
We do so by combining the ideas of \cite{facchinei1998accurate} and \cite{oberlinActiveSetIdentification2006}.

\begin{definition}
    We define the Lagrangian of \eqref{eq:POP} in the usual way as
    \begin{equation}
  \Lag(x, \lambda) \eqdef f(x) - \lambda^{\top} g(x),
\end{equation}
where $\lambda\in \dualSet$ is the Lagrange multiplier vector.
\end{definition}

First-order necessary conditions, or Karush-Kuhn-Tucker conditions (KKT), for $\opt$ to be a solution of \eqref{eq:POP}, under a qualification condition, are that there exist $\opt[\lambda]$ such that
\begin{equation}\label{eq:KKT}
  \begin{aligned}
  \nabla_{x} \Lag(\opt, \opt[\lambda]) &= 0, \\
  g_{\E}(\opt) &= 0, \\
  0 \le g_{\I}(\opt) &\perp \lambda_{\I}^{\star} \ge 0,
  \end{aligned}
\end{equation}
where $a \perp b$ means $a^{\top} b = 0$, which implies here that $g_{i}(\opt)\lambda_{i}^{\star} = 0$ for $i \in \I$.
The dual solution set is $\SolDual \eqdef \{ \opt[\lambda] \; \text{satisfying} \; \eqref{eq:KKT}\}$, while the primal dual solution set is $\Sol \eqdef \{\opt\} \times \SolDual$.

\paragraph{Qualification condititons}
We work in this section under the classical Mangasarian-Fromovitz constraint qualification (MFCQ): there is a vector $v\in\bbR^{n}$ such that
\begin{equation}\label{eq:MFCQ}
  \begin{aligned}
    \nabla g_{i}(\opt)^{\top} v < 0, \quad i\in\opt[\Actset]; \quad \nabla g_{i}(\opt)^{\top} v = 0, \quad i \in \E, \\
    \{\nabla g_{i}(\opt)\}_{i\in\E} \; \text{is linearly indepent.}
  \end{aligned}
\end{equation}
This is a mild assumption, that holds for generic problem coefficients.
This condition ensures that, if $\opt$ is a local minimizer, then it satisfies the KKT conditions \eqref{eq:KKT}, and that the dual solution set $\SolDual$ is bounded.

\paragraph{Active set estimation}
We consider the following measure of the KKT residual at point $(x, \lambda)$:
\begin{equation}
  \idfun(x, \lambda) \eqdef \left \|\nabla_{x} \Lag(x,\lambda) \right \| + \|g_{\E}(x)\| + \| \left[-g_{\I}(x)\right]_{+}\| + |\lambda_{\I}^T g_{\I}(x)|  +  \|[-\lambda_\I]_{+}\|,
\end{equation}
and recall next an error-bound result for analytic optimization problems \cite[Th. 5.1]{luoErrorBoundsAnalytic1994}.
\begin{proposition}\label{prop:errorbound}
  For each compact subset $\Omega\subset \bbR^{n}\times\dualSet$, there exists constants $\tau>0$ and $\gamma>0$ such that, for all $(x, \lambda) \in \Omega$, $\dist((x, \lambda), \Sol) \le \tau \idfun(x, \lambda)^{\gamma}$.
\end{proposition}

From a primal point $x$, we propose to build multipliers and active set as follows:
\begin{subequations}
\label{eq:identif}
\begin{align}
    \omega(x) &\eqdef \min_{\lambda \in \dualSet} \idfun(x, \lambda), \label{eq:identifpb}\\
    \Actset(x) &\eqdef \{i : g_{i}(x) \le -1 / \log(\omega(x))\}. \label{eq:activeset}
\end{align}
\end{subequations}
Note that \eqref{eq:identifpb} is a conic program with linear and second-order cone constraints, readily solved with interior point methods.

The following result shows that the above procedure correctly identifies a minimizer's active constraints from any neighboring point.
It requires only the lightest assumptions on the minimizer, which holds in particular when several dual solutions exist ($\SolDual$ not a singleton), or when strict complementarity fails ($g_{i}(\opt) = \lambda^{\star}_{i} = 0$ for some $i\in\I$).
\begin{theorem}[Active set identification]%
  \label{th:identif}
  Suppose that point $\opt$ satisfies the KKT \eqref{eq:KKT}, and the MFCQ \eqref{eq:MFCQ} conditions. %
  Then, we have the following:
  \begin{enumerate}
    \item the minimizers of \eqref{eq:identifpb} converge to $\SolDual$ as $x$ goes to $\opt$,
    \item there exists $\epsilon_{1}>0$ such that $x\in\ball(\opt, \epsilon_{1})$ implies $\|x-\opt\| \le \tau \omega(x)^{\gamma}$,
    \item there exists $\epsilon_{2}>0$ such that $x\in\ball(\opt, \epsilon_{2})$ implies $\Actset(x) = \opt[\Actset]$,
  \end{enumerate}
\end{theorem}
\begin{proof}
  \noindent\emph{Item i)}
  For $\opt[\lambda]\in\SolDual$ and $x$ near $\opt$, there holds
  \begin{align}
    \omega(x) &\le \idfun(x, \opt[\lambda]) \\
         &= \left \|\nabla_{x} \Lag(x,\opt[\lambda]) \right \| + \|g_{\E}(x)\| + \| \left[-g_{\I}(x)\right]_{+}\| + |\lambda_{\I}^{\star T} g_{\I}(x)|  +  \|[-\opt[\lambda]_\I]_{+}\| \\
         &= \begin{multlined}[t]
           \left \|\nabla_{x} \Lag(x,\opt[\lambda]) - \nabla_{x} \Lag(\opt,\opt[\lambda]) \right \| + \|g_{\E}(x) - g_{\E}(\opt)\| \\+ \| \left[-g_{\I}(x)\right]_{+} - \left[-g_{\I}(\opt)\right]_{+}\| + |\lambda_{\I}^{\star T} \left(g_{\I}(x) - g_{\I}(\opt)\right)|
         \end{multlined}\\
         &\le C \|x - \opt \|, \label{eq:idfunbound}
  \end{align}
  for some constant $C>0$, where we used that $\opt$ is a KKT point, the functions are Lipschitz continuous near $\opt$, and the set of optimal multipliers $\SolDual$ is bounded, as a consequence of MFCQ.

  We turn to show that, for $\eta>0$ small enough, there exists $\epsilon>0$ such that if $x\in\ball(\opt, \epsilon)$, then the minimizer of \eqref{eq:identifpb} belongs to $\SolDual + \eta$.
  We proceed by contradiction and assume that there exists two sequences $\curr$, $\curr[\lambda]$ such that $\curr\to\opt$, $\curr[\lambda]$ minimizes \eqref{eq:identifpb}, and $\dist(\curr[\lambda], \SolDual) > \eta$.
  By the above inequality, $\idfun(\curr, \curr[\lambda]) \le C\|\curr-\opt\|$, so that $\idfun(\curr, \curr[\lambda]) \to 0$.

  We first consider the case $\curr[\lambda]$ unbounded.
  Possibly taking a subsequence, we can assume that $\|\curr[\lambda]\| \to \infty$, $\curr[\lambda] / \|\curr[\lambda]\| \to \adh[\lambda]$, and $\|\adh[\lambda]\| = 1$ for some $\adh[\lambda] \in \dualSet$.
  Since $\idfun(\curr, \curr[\lambda]) \to 0$, we first have that $\lambda_{\I, k}^{\top}g_{\I}(\curr) \to 0$.
  Dividing by $\|\curr[\lambda]\|$ and taking limit yields $\sum_{i\in \I} \adh[\lambda]_{i} g_{i}(\opt) = 0$.
  Using that $\adh[\lambda] \in \bbR^{\mI}_{+}$, $g_{\opt[\Actset]}(\opt) = 0$, and $g_{i \in \I \setminus \opt[\Actset]}(\opt) > 0$, the above limit reduces towe get $\adh[\lambda]_{i} = 0$ for $i \in \I \setminus \opt[\Actset]$.
  The limit $\idfun(\curr, \curr[\lambda]) \to 0$ also implies $\nabla_{x} \Lag(\curr, \curr[\lambda]) \to 0$.
  Dividing by $\|\adh[\lambda]\|$, taking limit, and using $\adh[\lambda]_{\I\setminus\opt[\Actset]} = 0$ yields
  \begin{equation}
    \sum_{i\in\E\cup\opt[\Actset]}\adh[\lambda]_{i} \nabla g_{i}(\opt) = 0.
  \end{equation}
  We now use MFCQ \eqref{eq:MFCQ}: taking scalar product with vector $v$ and using conditions readily yields $\adh[\lambda]_{\opt[\Actset]}=0$ and then $\adh[\lambda]_{\E} = 0$, which contradicts the assumption.

  We turn to the case $\curr[\lambda]$ bounded.
  Any limit point $\adh[\lambda]$ of $\curr[\lambda]$ is such that $\idfun(\opt, \adh[\lambda]) = 0$, so that $\adh[\lambda] \in \SolDual$, which contradicts the assumption.

  \medskip
  \noindent\emph{Item ii)}
  For $\|x - \opt\| \le \epsilon$, we have by \emph{item i)} that $\dist(\lambda, \SolDual) \le \eta$.
  \Cref{prop:errorbound} provides, for any $\lambda$ in $\SolDual + \eta$:
  \begin{equation}
    \|x - \opt\| \le \dist((x, \lambda), \Sol) \le \tau \idfun(x, \lambda)^{\gamma},
  \end{equation}
  Taking first the infimum over $\lambda$, and then using \emph{item i)} yields:
  \begin{equation}
    \|x - \opt\| \le \tau \left(\min_{\lambda \in \SolDual}\idfun(x, \lambda)\right)^{\gamma} = \tau \left(\min_{\lambda \in \dualSet}\idfun(x, \lambda)\right)^{\gamma} = \tau \omega(x)^{\gamma}.
  \end{equation}

  \medskip
  \noindent\emph{Item iii)} Take $\epsilon_{2} > 0$ such that $\epsilon_{2} \le \epsilon_{1}$, $-1/\log(\omega(x)) \le \min_{i\notin\opt[\Actset]} g_{i}(\opt)/2$, and $L\tau\omega(x)^{\gamma} \le -1/\log(\omega(x))$.
  For $i \notin \opt[\Actset]$ and $x \in \ball(\opt, \epsilon_{2})$, we have
  \begin{equation}
    g_{i}(x) \le g_{i}(\opt) / 2 \le -1 / \log(\omega(x)),
  \end{equation}
  so that $i\notin \Actset(x)$.
  For $i \in \opt[\Actset]$ and $x \in \ball(\opt, \epsilon_{2})$,
  \begin{equation}
    -g_{i}(x) \le L \|x - \opt\| \le L\tau \omega(x)^{\gamma} \le -1 / \log(\omega(x)),
  \end{equation}
  where we used that $g_{i}$ is Lipschitz near $\opt$, and \emph{item ii)}.
  Therefore, $i \in \Actset(x)$.
\end{proof}

\begin{remark}[Relation with the literature]
  Identifying constraints in constrained optimization is a delicate topic and has attracted considerable attention.

  \Cref{th:identif} follows closely the approach of Oberlin and Wright \cite[Sec 3.2, Th. 3.4]{oberlinActiveSetIdentification2006}.
  Considering the wider setting of $\mathcal{C}^{2}$ rather than polynomial objective and constraint functions, they rely on a different error bound than that of \cref{prop:errorbound}.
  The ensuing identification procedure requires solving a linear program with equilibrium constraints, which is done with a branch and bound strategy on a combinatorial program.
  In contrast, the error bound used here yields a conic optimization problem, easily solved by Interior points methods.

  Lewis and Wright \cite{lewisIdentifyingActivity2011} look at $\min g \circ c$ where $c$ is a smooth mapping, and $g$ is nonsmooth.
  This captures problem \eqref{eq:POP}, but also the larger class of nonsmooth functions.
  They characterize the identification properties of abstract algorithmic frameworks, assuming that some sequence of points converge to a minimizer.
\end{remark}

\begin{example}[Activity identification]%
  \label{ex:toypbdidentif}
  We return to the problem developed in \cref{ex:toypb,ex:toypb-moment}.
  \Cref{fig:toypb_identif_socp} illustrates the behavior of the identification procedure \eqref{eq:identif} near the minimizer $\opt$.
  In line with \cref{th:identif}, the active set $\Actset = \{1\}$ is consistently detected near the minimizer.

  Therefore, the polynomial problem \eqref{eq:POP} reduces to the equality constrained minimization
  \begin{equation}%
    \label{eq:toypbreducedpop}
    \min_{x} f(x) \quad \text{s.t.} \quad g_1(x) = 0,
  \end{equation}
  whose resolution is the topic of the following section.
\end{example}

\begin{figure}[tbp]
  \centering
  \subfloat[]{\label{fig:toypb_identif_socp_a}\includegraphics[width=0.48\textwidth, valign=c]{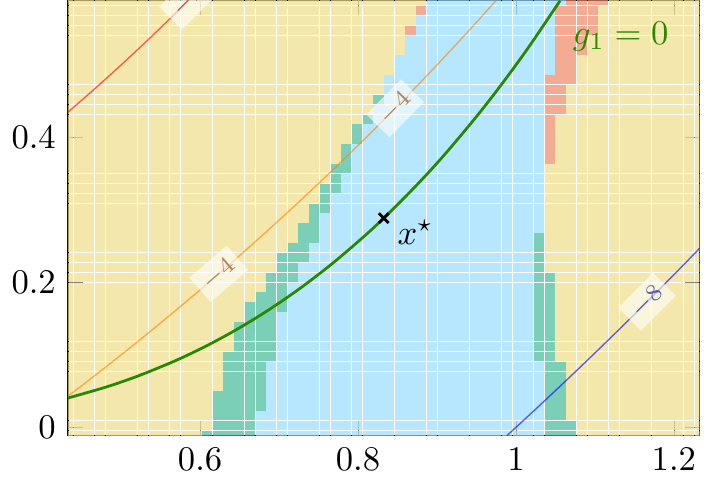}}
  \qquad
  \subfloat[]{\label{fig:toypb_identif_socp_b}\includegraphics[width=0.18\textwidth, valign=c]{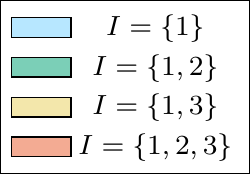}}
  \caption{Behavior of the identification procedure \cref{eq:identif} on the illustrative problem of \cref{ex:toypb}: the correct active set $I = \{1\}$ is consistently detected on a neighborhood of $\opt$.}
  \label{fig:toypb_identif_socp}
\end{figure}

\section{Algebraic Geometry and Point Estimation Theory}
\label{sec:algebr-geom-point}

In this section, we review the so-called point-estimation theory of Smale~\cite{ShubSmale1993,Cucker1999}, also known as the \ab{} theory.
This theory allows to capture the behavior of Newton's method quantitatively, and provides a computable criterion that guarantees existence of a nearby zero and quadratic convergence.
We refer to~\cite{alvarezUnifyingLocalConvergence2008} for an extension of this result beyond the algebraic setting.

Throughout this section, we consider a system of polynomial equations $\Psys:\bbR^{\Psysn}\to\bbR^{\Psysm}$, and Newton's method defined as iterating $N_{\Psys}(\vPsys) \eqdef [\D_{\Psys}(\vPsys)]^{-1}\Psys(\vPsys)$ for $\vPsys \in \bbR^{\Psysn}$.

\begin{definition}[Approximate Zero]
  Let $\vPsys \in \R^{\Psysn}$ and consider the sequence $\vPsys_0 = \vPsys$, $\vPsys_{i + 1} = N_{\Psys}(\vPsys_i)$ for $i \ge 0$.
  The point $\vPsys$ is an \emph{approximate zero} of $\Psys$ if this sequence is well defined, and there exists a point $\vPsys' \in \R^\Psysn$ such that $\Psys(\vPsys') = 0$ and
  \begin{align}%
    \label{eq:newtonquadcv}
    \|\vPsys_i - \vPsys' \| \le (1/2)^{2^i - 1} \|\vPsys_0 - \vPsys'\|.
  \end{align}
  Then, we call $\vPsys'$ the \emph{associated zero} of $\vPsys$ and say that $\vPsys$ represents $\vPsys'$.
\end{definition}
Besides, we will need the following quantities.
\begin{subequations}
  \begin{align}
    \alpha (\Psys, \vPsys) &\eqdef  \beta (\Psys, \vPsys) \gamma(\Psys, \vPsys) \\
    \beta (\Psys, \vPsys) &\eqdef \left\| \D_\Psys(\vPsys)^{-1}  \Psys(\vPsys)  \right\| \\
    \gamma(\Psys, \vPsys) &\eqdef \sup\limits_{k > 1} \left\| {\frac{\D_\Psys(\vPsys)^{-1}    \D_\Psys^{(k)} (\vPsys)}{k!}} \right\|^{\frac{1}{k-1}}.
  \end{align}
\end{subequations}

The following result shows that if $\alpha(\Psys, \vPsys)$ is small enough, there exists a nearby zero to which the iterates of Newton's method converge quadratically.

\begin{proposition}[Detecting approximate zeros]%
  \label{prop:alphabeta}
  Consider any polynomial map $\Psys:\bbR^{\Psysn} \to \bbR^{\Psysm}$ and point  $\vPsys\in\bbR^{\Psysn}$. If $\alpha(\Psys, \vPsys) \le \alpha_0$, then
  \begin{enumerate}
    \item $\vPsys$ is an approximate zero of $\Psys$,
    \item $\|\vPsys - \vPsys'\| \le 2 \beta(\Psys, \vPsys)$,
  \end{enumerate}
  where $\vPsys'$ is the associated zero of $\vPsys$.
  The universal constant $\alpha_{0}$ is smaller than  $\frac{1}{4}(13 - 3 \sqrt{17}) \approx 0.158$.
\end{proposition}

The quantity $\gamma(\Psys, \vPsys)$ involves the operator norm of a sequence of increasingly large tensors.
Its direct computation is therefore hardly practical.
However,~\cite{ShubSmale1993} proposed a tractable upper bound for this quantity, when $\Psysn = \Psysm$.

\begin{proposition}[Bound on high derivatives]%
  \label{prop:tractablebound}
  Consider a polynomial mapping $\Psys: \R^{\Psysn}\to\R^{\Psysn}$ and a point $\vPsys\in\bbR^{\Psysn}$.
  If $\D_{\Psys}(\vPsys)$ is invertible, there holds
 \begin{equation}
    \gamma (\Psys, \vPsys) \leq \mu(\Psys, \vPsys) \frac{(\max_{i=1, \ldots, \Psysn} d_{i})^{3/2}}{2\|\vPsys\|_1},
  \end{equation}
  where $d_{i} = \degree \Psys_{i}$, $\mu(\Psys, \vPsys) \eqdef \max \{1, \|\Psys\|_{p} \|D_\Psys(\vPsys)^{-1} \Delta_{(d)}(\vPsys)\|\}$, and $\Delta_{(d)}(\vPsys)$ is the $\Psysn\times \Psysn$ diagonal matrix whose $i$-th element is $\Delta_{(d)}(\vPsys)_{i,i} = d_i^{1/2}\|\vPsys\|_1^{d_i-1}$.
\end{proposition}

In practice, we thus use the following direct consequence of \cref{prop:alphabeta,prop:tractablebound}.
\begin{corollary}[Tractable detection of approx. zeros]\label{coro:ab}
  Consider a polynomial map $\Psys$ and a point $\vPsys$.
  If
  \begin{equation}
    \hat{\alpha}(\Psys, \vPsys) \eqdef \beta(\Psys, \vPsys) \mu(\Psys, \vPsys) \frac{(\max_{i=1, \ldots, \Psysn} d_{i})^{3/2}}{2\|\vPsys\|_1} \le \frac{1}{4}(13-3\sqrt{17}),
  \end{equation}
  then \emph{(i)} $\vPsys$ is an approximate zero of $\Psys$, and \emph{(ii)}$\|\vPsys - \vPsys'\| \le 2 \beta(\Psys, \vPsys)$,
  where $\vPsys'$ is the associated zero of $\vPsys$.
\end{corollary}

\begin{example}[Newton on the reduced problem]
  We return to the problem discussed in \cref{ex:toypb,ex:toypb-moment,ex:toypbdidentif}.
  Having obtained a point $\vx$ near $\opt$, and estimated its active set $\Actset(\vx) = \{1\}$, the initial problem \eqref{eq:toypbfuns} reduces to \eqref{eq:toypbreducedpop}.
  This reduced problem features only equality constraints, so that its KKT conditions form the system of smooth equations:
  \begin{align}
      \nabla f(x) - \lambda \nabla g_{1}(x) &=0\\
      g_{1}(x) &=0
  \end{align}
  We consider solving this system of equations with Newton's method, relying on the \ab{} test to ensure fast convergence of Newton's method.
\end{example}

\section{A hybrid method for polynomial optimization}
\label{sec:hybr-meth-polyn}

In this section, we present a method that combines the three ingredients into an algorithm for global optimization of polynomial problems.

\subsection{Description of the algorithm}
\label{sec:descr-algor}

The proposed \cref{alg:hybrid} consists of an outer-inner loop structure, where the outer loop produces an estimate of the primal solution, and the inner loop potentially applies a fast local convergent procedure that either converges to the solution or is forced to step back out in the outer loop upon failure of feasibility or optimality.

\paragraph{Partial solve of convex relaxation}
A convex relaxation of the polynomial problem is partially solved with, \eg{} a first-order method such as coordinate descent on the Burer-Monteiro reformulation of the relaxation as described in \cite{marecekLowrankCoordinatedescentAlgorithm2017}.

\paragraph{Active set identification}
At the obtained point $\curr$, the current active set $\curr[\Actset] \subseteq \I$ is computed with the procedure \cref{eq:activeset}.

\paragraph{Newton's method on the reduced problem}
We then consider the problem reduced to the current active set $\curr[\Actset]$, defined by
\begin{equation}%
  \label{eq:reducedPOP}
  \min_{x} f(x) \quad \text{s.t.} \quad g_i(x) = 0, \quad i \in \E \cup \curr[\Actset]
\end{equation}
We apply Newton's method on the first-order necessary optimality, or KKT, conditions to find a local minimizer.
In line with the \ab{} theory, this writes as finding a zero of the polynomial system $\Psys_{\curr[\Actset]}: \bbR^{n +\mE + |\curr[\Actset]|} \to \bbR^{n +\mE + |\curr[\Actset]|}$, defined by
\begin{equation}%
  \label{eq:redpopKKT}
  \Psys_{\curr[\Actset]}(x, \lambda) \eqdef
  \begin{pmatrix}
    \nabla f(x)-\sum_{i\in \E \cup \curr[\Actset]} \lambda_{i} \nabla g_i(x) \\
    g_{\E \cup \curr[\Actset]}(x)
  \end{pmatrix}.
\end{equation}
We consider running Newton's method from the point $(\curr, \lambda_{\curr[\Actset]}(\curr))$, where the dual point is the least-squares, or Fletcher's, multiplier defined as:
\begin{equation}\label{eq:leastsquaremult}
  \lambda_{\Actset}(x) \eqdef \min_{\lambda\in\bbR^{\mE + |\curr[\Actset]|}} \left\| \nabla f(x) - \sum_{i\in \E \cup \curr[\Actset]} \lambda_{i}\nabla g_{i}(x) \right\|.
\end{equation}
If the \ab{} test is satisfied, we run Newton's method up to high precision: $\|z - \opt{z}\|\le \varepsilon$.
By the estimate \cref{eq:newtonquadcv}, this takes $\lceil\ln_{2}( 1 - \ln_{2}(\varepsilon/\|\vPsys - \opt[\vPsys]\|) )\rceil$ iterations: with the classical floating point accuracy $\varepsilon=2.22\cdot10^{-16}$, and starting with initial distance $\|\vPsys - \opt[\vPsys]\| = 10$, Newton's method converges in at most $\lceil5.82\rceil = 6$ iterations.

\paragraph{Checking global optimality}
No global approach.
For specific problems, one may have a good quality lower bound on the optimal value $f(\opt)$, so that for feasible iterates $\curr$, $f(\curr)-f(\opt)$ bounds the suboptimality $f(\curr)-f(\opt)$.
This is the case in maxcut problems.
In generality, one possibility is to check that the point $\adh$ gives a sequence of moments $(\adh^{\alpha})_{\alpha}$ that is optimal for a tight enough moment relaxation.
This task amounts to checking optimality of a feasible point for an SDP in dual form, without a corresponding candidate point for the primal problem.
Doing so with less effort than solving the primal dual pair is challenging; an approach that leverages the specifics of the current situation is outlined in \cite{xuVerifyingGlobalOptimality2021}.

\begin{algorithm}
  \caption{Hybrid algorithm -- Two phase algorithm}
  \label{alg:hybrid}
  \begin{algorithmic}[1]
    \Require $(f, g)$ defining the POP \eqref{eq:POP}%
    \State Initialise $r = \max(d_{i})$, $y^{0}_{r}$ initial point for \eqref{eq:sdprel}
    \For{$k=1,2,\dots$}
      \State $y^{k}_{r} \gets $ partial solution of the order-$r$ moment relaxation \eqref{eq:sdprel} from point $y^{k-1}_{r}$
      \State $x^{k} \gets $ extracted point from SDP iterate $y^{k}_{r}$ \label{line:extraction}%
      \label{alg:lextract}
      \label{alg:lextr}
      \State $\Actset^{k} \gets \Actset(x^{k})$ %
      \State $\lambda^{k} \gets \lambda_{\Actset^{k}}(x^{k})$ %
      \label{alg:lactset}
      \If{$\hat{\alpha}\left(\Psys_{\Actset^{k}}, (x^{k}, \lambda^{k})\right) \le \frac{1}{4}(13 - 3 \sqrt{17})$}
        \State $\adh[x]^{k}, \adh[\lambda]^{k} \gets$ limit of Newton's method on $\Psys_{\Actset^{k}}$ from point $(x^{k}, \lambda^{k})$
          \label{alg:lnewton}
        \If{$(\adh[x]^{k}, \adh[\lambda]^{k})$ satisfies a condition for global optimality} \label{alg:lstopttest}
          \State \Return $\adh[x]^{k}, \adh[\lambda]^{k}$
          \label{alg:lnewtonreturn}
        \Else %
          \State Further solve current relaxation or increase $r$
        \EndIf
      \EndIf
    \EndFor
  \end{algorithmic}
\end{algorithm}

\subsection{Convergence guarantees for \Cref{alg:hybrid}}
\label{sec:conv-alg}

We introduce two assumptions on the problem at the minimizer $\opt$.
The linear independence constraint qualification (LICQ) is that
\begin{equation}\label{eq:LICQ}
  \{ \nabla g_{i}(\opt), \; i \in \E\cup\opt[\Actset]\} \; \text{is linearly indenpendant}.
\end{equation}
This condition implies the Mangasarian-Fromovitz condition \eqref{eq:MFCQ} and ensures unicity of the set multipliers at $\opt$; see \eg{} \cite[p. 202]{bonnansNumericalOptimizationTheoretical2006}.
The second-order sufficient condition is that
\begin{equation}\label{eq:secondordergrowth}
  v^{\top} \nabla_{xx}^{2} \Lag(\opt, \opt[\lambda]) v >0 \; \text{for all} \; v \in \mathcal{C}\setminus\{0\},
\end{equation}
where $\opt[\lambda]$ is the dual multiplier, and
\begin{equation}
  \mathcal{C} \eqdef \{v \; | \; \nabla g_{i}(\opt)^{\top} v = 0, \;\text{for}\; i\in\E\cup\opt[\Actset]\}.
\end{equation}

We can now state the main result on \cref{alg:hybrid}.
\begin{theorem}[Quadratic convergence to the global minimizer]%
  \label{th:algocv}
  Consider a POP~\eqref{eq:POP} that admits a unique global minimizer $\opt$, and assume that the solver for solving the relaxation \eqref{eq:sdprel} generates iterates converging to the solution.
  Then, \cref{alg:hybrid} generates iterates $(\curr, \curr[\Actset])$ such that:
  \begin{enumerate}
    \item $(\curr)$ converges to $\opt$.
  \end{enumerate}
  If $\opt$ also satisfies the LICQ \eqref{eq:LICQ}, and the second-order growth condition \eqref{eq:secondordergrowth}, then
  \begin{enumerate}
      \setcounter{enumi}{1}
    \item there exists $\hat{k}$ such that the minimizer active set is identified $\Actset_{\hat{k}} = \opt[\Actset]$ (l.~\ref{alg:lactset}), the Newton procedure (l.~\ref{alg:lnewton}) converges quadratically, and it returns the optimal primal-dual point $(\adh[x]^{\hat{k}}, \adh[\lambda]^{\hat{k}})$.
  \end{enumerate}
\end{theorem}

Some comments are in order:
\begin{itemize}
\item The uniqueness of the minimizer allows for a simple way to convert the SDP iterates $y_{r}^{k}$ into a point $x^{k}$ in the POP space (line~\ref{line:extraction}), presented in \cref{prop:momentrel}.
    Using more complex tools, such as the Gelfand–Naimark–Segal method \cite{henrion2005detecting,klepMinimizerExtractionPolynomial2018}, or a homotopy continuation method \cite{weisser2019polynomial}, one can weaken this assumption to existence of multiple \emph{distinct} minimizers and preserve the guarantees of \cref{th:algocv}.
  \item The LICQ and second-order growth conditions guarantee the existence of a neighborhood of $\opt$ on which Newton's method converge quadratically, by notably ensuring uniqueness of the optimal multiplier.
    This requirement can be weakened by using a \emph{stabilized} Newton iteration in lieu of the plain Newton iteration for solving the reduced problem~\eqref{eq:reducedPOP}; see \eg \cite{wrightAlgorithmDegenerateNonlinear2005,gillStabilizedSQPMethod2017} and references therein.
\end{itemize}

Before going to the proof of \cref{th:algocv}, we need the following lemma.
\begin{lemma}[Guaranteed quadratic convergence]%
  \label{lmm:quadcvreducedpb}
  Assume that the minimizer $\opt$ satisfies the LICQ \eqref{eq:LICQ} and second-order growth condition \eqref{eq:secondordergrowth}, and consider the system of KKT equations $\Psys_{\opt[\Actset]} = 0$ \eqref{eq:redpopKKT} of the reduced problem \eqref{eq:reducedPOP}.
  There exists a neighborhood $\nhd_{\opt, \opt[\lambda]}$ of $(\opt, \opt[\lambda])$ over which the \ab{} condition is satisfied:
  \begin{equation}
    \hat{\alpha}(\Psys_{\opt[\Actset]}, \cdot) \le \frac{1}{4}(13-3\sqrt{17}) \approx 0.15 \quad \text{for all } x,\lambda \in \nhd_{\opt, \opt[\lambda]}.
  \end{equation}
\end{lemma}

\begin{proof}
  First, recall that
  \begin{equation}
    \hat{\alpha}(\Psys, \vPsys) = \left\| \D_\Psys(\vPsys)^{-1}  \Psys(\vPsys)  \right\| \frac{(\max_{i=1, \ldots, \Psysn} d_{i})^{3/2}}{2\|\vPsys\|_1} \max \{1, \|\Psys\|_{p} \|D_\Psys(\vPsys)^{-1} \Delta_{(d)}(\vPsys)\|\}.
  \end{equation}
  We have that $\|\cdot\|_{1}\ge 1$, $\|\vPsys\|_{1}$ is bounded near $(\opt, \opt[\lambda])$, and that $\Delta_{(d)}(\vPsys)$ is a diagonal matrix with entries $d_i^{1/2}\|\vPsys\|_1^{d_i-1}$.
  Therefore, there exist positive constants $C_{1}$ and $C_{2}$ such that:
  \begin{equation}
    \hat{\alpha}(\Psys, \vPsys) \le C_{1} \| \D_\Psys(\vPsys)^{-1}\|  \|\Psys(\vPsys)\| \max \{1, C_{2} \|D_\Psys(\vPsys)^{-1} \|\}.
  \end{equation}
  Furthermore, the term $\|\D_{\Psys}(\cdot)\|$ is bounded near $\opt[z]$.
  Indeed,
  \begin{equation}
    \D_{\Psys}(z) =
  \begin{pmatrix}
    \nabla^{2}_{xx} \Lag(z) & \Jac_{g_{\E \cup \opt[\Actset]}}(z_{x}) \\ %
    \Jac_{g_{\E \cup \opt[\Actset]}}(z_{x}) & 0
  \end{pmatrix},
  \end{equation}
  where $\Jac_{g_{\E \cup \opt[\Actset]}}(z_{x})$, the jacobian of the active constraints at the primal point $z_{x}$, is full-rank by LICQ, and $\nabla^{2}_{xx} \Lag(z)$ is invertible when restricted to $\ker \Jac_{g_{\E \cup \opt[\Actset]}}(z_{x})$, by the second-order growth condition.
  Therefore, $\D_{\Psys}(\opt[z])$ is invertible by \cite[Prop. 14.1]{bonnansNumericalOptimizationTheoretical2006}.
  By continuity of the mappings, there exist a neighborhood of $(\opt, \opt[\lambda])$ over which $\|\D_{F}(\cdot)^{-1}\|$ is bounded.
  The result follows, since $F(x, \lambda)$ goes to zero as $(x, \lambda)$ go to $(\opt, \opt[\lambda])$.
\end{proof}

\begin{proof}[Proof of \cref{th:algocv}]
  \emph{Item i)}
  First, if the global optimality test at line~\ref{alg:lstopttest} is never satisfied, then \cref{alg:hybrid} solves with growing accuracy relaxations of growing order.
  Therefore, the extracted points $\curr$ converge to $\opt$ by \cref{prop:momentrel}.
  In this case, the sequence $\curr$ is not affected by lines~\ref{alg:lextr} to~\ref{alg:lnewtonreturn}.
  Otherwise, if the test at line~\ref{alg:lstopttest} is satisfied, then algorithm returns the minimizer after a finite number of iterations.

  \emph{Item ii)}
  \Cref{th:identif} provides the existence of a neighborhood $\nhd_{\opt}$ of the minimizer $\opt$ over which its active set is identified $\curr[\Actset] = \opt[\Actset]$. The theorem applies: since point $\opt$ satisfies the LICQ \eqref{eq:LICQ}, the MFCQ \eqref{eq:MFCQ} and KKT \eqref{eq:KKT} also hold \cite[p. 202]{bonnansNumericalOptimizationTheoretical2006}.
  Besides, \cref{lmm:quadcvreducedpb} shows that, for the reduced problem~\eqref{eq:reducedPOP} with $\opt[\Actset] = \opt[\Actset]$, there exists a neighborhood $\nhd_{\opt[z]} \subset \bbR^{n}\times\bbR^{|\opt[\Actset]|}$ of the primal-dual optimal point $\opt[z] = (\opt, \opt[\lambda])$ for which the \ab{} test is positive, and thus from which Newton's method converges quadratically.
  Finally, the initial dual point of Newton's method is defined as the least square multiplier $\lambda_{\Actset}$ in \eqref{eq:leastsquaremult}.
  It is characterized by its optimality condition:
  \begin{align}
    \lambda_{\Actset}(x) = - [\Jac_{h_{\Actset}}(x) \Jac_{h_{\Actset}}(x)^{\top}]^{-1} \Jac_{h_{\Actset}}(x) \nabla f (x),
  \end{align}
  where the mapping $h_{\Actset}:\bbR^{n}\to\bbR^{\mE + |\Actset|}$ is the concatenation of the equality and active inequality constraints mappings.
  When $\Actset=\opt[\Actset]$, the mapping $h_{\opt[\Actset]}$ is smooth, and the LICQ condition implies that its Jacobian is full-rank near $\opt$.
  Therefore, $\lambda_{\opt[\Actset]}$ is a Lipschitz continuous function near $\opt$.

  Putting the pieces together, since the iterates $(x^{k})$ converge to $\opt$ by item \emph{i)}, $x^{k}$ eventually belongs to $\nhd_{\opt}$.
  For all following iterations, the correct active set is estimated.
  By continuity of the multipliers, the pair $(\curr, \curr[\lambda])$ converges to $(\opt, \opt[\lambda])$, so that it eventually belongs to $\nhd_{\opt[z]}$, from which point Newton's method converges quadratically to $\opt, \opt[\lambda]$.
\end{proof}

\section{Experimental results}
\label{sec:experiments}

In this section, we demonstrate the applicability of the hybrid methodology for solving polynomial optimization problems.
We first review the Optimal Power Flow problem \cite{cain2012history,josz2014application,ghaddarOptimalPowerFlow2016}, and then present and discuss the behavior of the hybrid method on IEEE test cases.

\subsection{ACOPF as a Polynomial Optimization Problem}
\label{sec:acopf-as-polynomial}

In this section, we present a comprehensive evaluation on the so-called rectangular power-voltage formulation of the Alternative Current Optimal Power Flow (ACOPF).
There, the electric power transmission system is represented by an undirected graph with vertices $N$ (called buses) and edges $E\subseteq N\times N$ (called branches).
Some buses, $G\subseteq N$, are called generators.
The total number of buses is $n= | N |$.
The variables are:
\begin{itemize}
  \item $P_k^g + j Q_k^g \in \mathbb{C}$: total power generated at bus $k\in G$.
  \item $V_k \eqdef \mfR{V_k} + j \mfI{V_k} \in \mathbb{C}$: voltage at bus $k\in G$.
  \item $S_{lm} \eqdef P_{lm} + j Q_{lm}\in \mathbb{C}$: the apparent power flow on branch $(l,m)\in E$.
\end{itemize}

A number of constants are in use:
\begin{itemize}
  \item $P_k^d + j Q_k^d \in \mathbb{C}$: active and reactive load at each bus $k\in N$.
  \item $y\in \mathbb{C}^{n\times n}$: the network admittance matrix with the same sparsity pattern as the network.
  \item $\bar{b}_{lm} \in \mathbb{R}$: the value of the shunt element at branch $(l,m)\in E$.
  \item $g_{lm} + j b_{lm} \in \mathbb{C}$: series admittance at branch $(l,m)\in E$.
  \item $P_k^{min}, P_k^{max}, Q_k^{min}, Q_k^{max} \in \mathbb{R}$: limits on active and reactive generation capacity at bus $k\in N$; $P_k^{min}= P_k^{max}= Q_k^{min}= Q_k^{max} = 0$ for all $k\in N/G$.
  \item $V_k^{min}, V_k^{max} \in \mathbb{R}$: limits on absolute value of $V_k$.
  \item $S_{lm}^{max} \in \mathbb{R}$: limits on absolute value of the apparent power of branch $(l,m)\in E$.
\end{itemize}

The relations between the variables $P_k^g, Q_k^g, V_k, P_{lm}, Q_{lm}$ are represented in terms of voltages, as power generated at each bus $k\in G$:
\begin{align+}
  P_k^g &= P_k^d + \mfR{V_k} \sum_{i=1}^n( \mfR{y_{ik}} \mfR{V_i} - \mfI{y_{ik}}\mfI{V_i}) + \mfI{V_k} \sum_{i=1}^n( \mfI{y_{ik}} \mfR{V_i} + \mfR{y_{ik}}\mfI{V_i}),
          \label{Pkg}
  \\ %
  Q_k^g &= Q_k^d + \mfR{V_k} \sum_{i=1}^n( - \mfI{y_{ik}} \mfR{V_i} - \mfR{y_{ik}}\mfI{V_i}) +
          \mfI{V_k} \sum_{i=1}^n( \mfR{y_{ik}} \mfR{V_i} - \mfI{y_{ik}}\mfI{V_i}),
          \label{Qkg}
\end{align+}

and as power-flow equations at each branch $(l, m)\in E$:
\begin{equation+}
  \label{Plm}
  P_{lm} = b_{lm}( \mfR{V_l} \mfI{V_m} - \mfR{V_m}\mfI{V_l}) + g_{lm} \left( (\mfR{V_l})^2 + (\mfI{V_l})^2 - \mfI{V_l}\mfI{V_m}- \mfR{V_l}\mfR{V_m}\right),
\end{equation+}
\begin{align}
  \label{Qlm}
  Q_{lm} = \;&b_{lm} \left( \mfR{V_l} \mfI{V_m} - \mfI{V_l}\mfI{V_m} - (\mfR{V_l})^2 - (\mfI{V_l})^2\right) \\
  &+ g_{lm} ( \mfR{V_l} \mfI{V_m} - \mfR{V_m} \mfI{V_l} - \mfR{V_m}\mfI{V_l}) - \frac{\bar{b}_{lm}}{2}\left((\mfR{V_l})^2 + (\mfI{V_l})^2\right).
\end{align}

The objective is to minimize the cost of total power generation where the cost at generator $k\in G$ is denoted as $$f_k(P_k^g) = c_{k, 2} (P_k^g)^2 + c_{k,1} P_k^g + c_{k,0},$$ subject to the constraints \eqref{Pkg}$-$\eqref{Qlm} and certain simple additional constraints:
\begin{align}
  \min_{P_{k}^{g}, Q_{k}^{g}, V_{k}\in\mathbb{C}, P_{lm}, Q_{lm}} &\sum_{k\in G} \left( c_{k, 2} (P_k^g)^2 + c_{k,1} P_k^g + c_{k,0} \right) \tag{\textbf{ACOPF}}\label{P2}\\
  \st\ & P_k^{min} \leq P_k^g \leq P_k^{max}\\
        & Q_k^{min} \leq Q_k^g \leq Q_k^{max}\\
        & (V_{k}^{min})^2\leq (\mfR{V_k})^2 + (\mfI{V_k})^2 \leq (V_k^{max})^2\\
        & P_{lm}^2 +Q_{lm}^2 \leq (S_{lm}^{max})^2,\\
        &\eqref{Pkg}-\eqref{Qlm}.
\end{align}

This problem writes equivalently as a POP \eqref{eq:POP} in variable $x=(\Re V, \Im V) \in \bbR^{2N}$, where $Y_{k}$, $\bar{Y}_{k}$, $M_{k}$, $Y_{lm}$, $\bar{Y}_{lm}$ are $2n\times2n$ coefficient matrices as in \cite{ghaddarOptimalPowerFlow2016}.

\begin{align}
  \min_{x\in\bbR^{2N}} &\sum_{k\in G} \left( c_{k, 2} [P_k^g + \tr(Y_k xx^T)]^2 + c_{k,1} [P_k^g + \tr(Y_k xx^T)] + c_{k,0} \right) \tag{\textbf{POP}}\label{P3}\\
  \st\ & P_k^{min} \leq P_k^g + \tr(Y_k xx^T) \leq P_k^{max}\\
        & Q_k^{min} \leq Q_k^g + \tr(\bar{Y}_k xx^T) \leq Q_k^{max}\\
        & (V_{k}^{min})^2\leq  \tr(M_k xx^T) \leq (V_k^{max})^2\\
        & (\tr(Y_{lm} xx^T) )^2 +(\tr(\bar{Y}_{lm} xx^T) )^2 \leq (S_{lm}^{max})^2%
\end{align}

\subsection{An implementation}
\label{sec:comm-numer-exper}

We have implemented the following algorithms:
\begin{enumerate}
  \item coordinate descent of the first moment relaxation of \eqref{P3}, detailed in \cite{marecekLowrankCoordinatedescentAlgorithm2017};
  \item a globalized version of Newton directly on \eqref{P3}, \eg{} \cite[Alg. 4.3]{bonnansNumericalOptimizationTheoretical2006};
  \item the hybrid scheme \cref{alg:hybrid}, without the backtracking: we run coordinate descent on the first-order relaxation, then switch to Newton's method as soon as the \ab{} test is satisfied for the reduced problem.
\end{enumerate}
In order to study the importance of the \ab{} test, we will also consider the algorithm where coordinate descent is run for a fixed number of iterations before switching to Newton's method.

The \ab{} test has a simple expression for \eqref{P3}.
The polynomials in the KKT equations of reduced problems are of degree $3$, so that $\Delta_{(d)}(\vPsys) = 3^{1/2}\|\vPsys\|_1^2 I_{2n\times 2n}$, and
\begin{align}
\hat{\alpha}(\Psys, \vPsys) &= \beta(\Psys, \vPsys) \max \{1, \|\Psys\|_{p} \|D_\Psys(\vPsys)^{-1} \Delta_{(d)}(\vPsys)\|\} \frac{(\max_{i=1, \ldots, \Psysn} d_{i})^{3/2}}{2\|\vPsys\|_1} \\
                       &= \beta(\Psys, \vPsys) \max\left( \frac{1.5\sqrt{3}}{\|\vPsys\|_{1}}, 4.5\sqrt{2n}\frac{\|\Psys\|_{p} \|\vPsys\|_{1}}{\sigma} \right),
\end{align}
where the spectral norm of $\|D_\Psys(\vPsys)^{-1}\|$ is replaced by the inverse of $\sigma$, the smallest eigenvalue of $D_\Psys(\vPsys)$.
We employ this last expression in our experiments.

\subsection{Results}

We have tested our implementation on IEEE 30-bus, 118-bus, 300-bus, and 2383-bus test cases (case30, case118, case300, and case2383wp), as distributed with the MatPower library \cite{zimmermanMATPOWERSteadyStateOperations2011}.

Let us discuss the results in turn.
On the 30-bus IEEE test case, \cref{fig:case30cdNewton} demonstrates that switching early from coordinate descent (CD) to Newton's method (after 2, 4, or 8 epochs of the coordinate descent, cf. \texttt{CD 2}, \texttt{CD 4}, or \texttt{CD 8}) generates wild oscillations in objective function value and the cardinality of the active set. This results in a slow decrease in infeasibility.
On this particular instance, there is no discernible impact on the objective function value attained eventually, but this is not the case in general.

On the 118-bus IEEE test case, \Cref{fig:case118cdNewton} (center and right subfigures) illustrate that the objective function value attained
by a combination of a first-order method and second-order method
can vary substantially, depending on the number epochs of the coordinate descent (CD) performed, until the switch to the second-order method.
This illustrates that Newton's method may be attracted to (first-order stationary) points that are not the global minimizer, and highlights the need for a backtracking procedure and a global optimality check.
On the same instance, \Cref{fig:case118compare} (left plot) then confirms that the hybrid method with the \ab{} test improves upon either coordinate descent or the Newton method; the fast quadratic convergence of Newton's method near first-order critical points is illustrated by the sharp decrease of infeasibility and stabilization of function value.
For the sake of completeness, notice that \cref{fig:case118compare} also shows that away from critical points, Newton's method can be quite slow.
As a first-order method, coordinate descent on its own can be quite slow: it fails to produce any significant reduction of infeasibility in the allowed time in \cref{fig:case118compare}. %
On the 300-bus IEEE test case, in \cref{fig:case300} we observe a very similar performance to the 118-bus test case (\Cref{fig:case118compare}).

On the 2383-bus instance, which captures the transmission system of Poland at its winter peak,
we use a slightly more complicated model, as described in Section 5 of \cite{marecekLowrankCoordinatedescentAlgorithm2017},
across both the first- and second-order methods,
but the performance of the code is largely unaffected.
\Cref{fig:case2383wp} displays a very similar behaviour to \Cref{fig:case300} %
in terms of the convergence rate (bottom row in \Cref{fig:case2383wp}), although the absolute run-times are necessarily longer (top row in \Cref{fig:case2383wp}).
The proposed hybrid method compares favorably relative to coordinate descent and Newton's method on their own, especially relative to time.

\begin{figure*}[htbp]
  \center
  \includegraphics[width=0.32\textwidth]{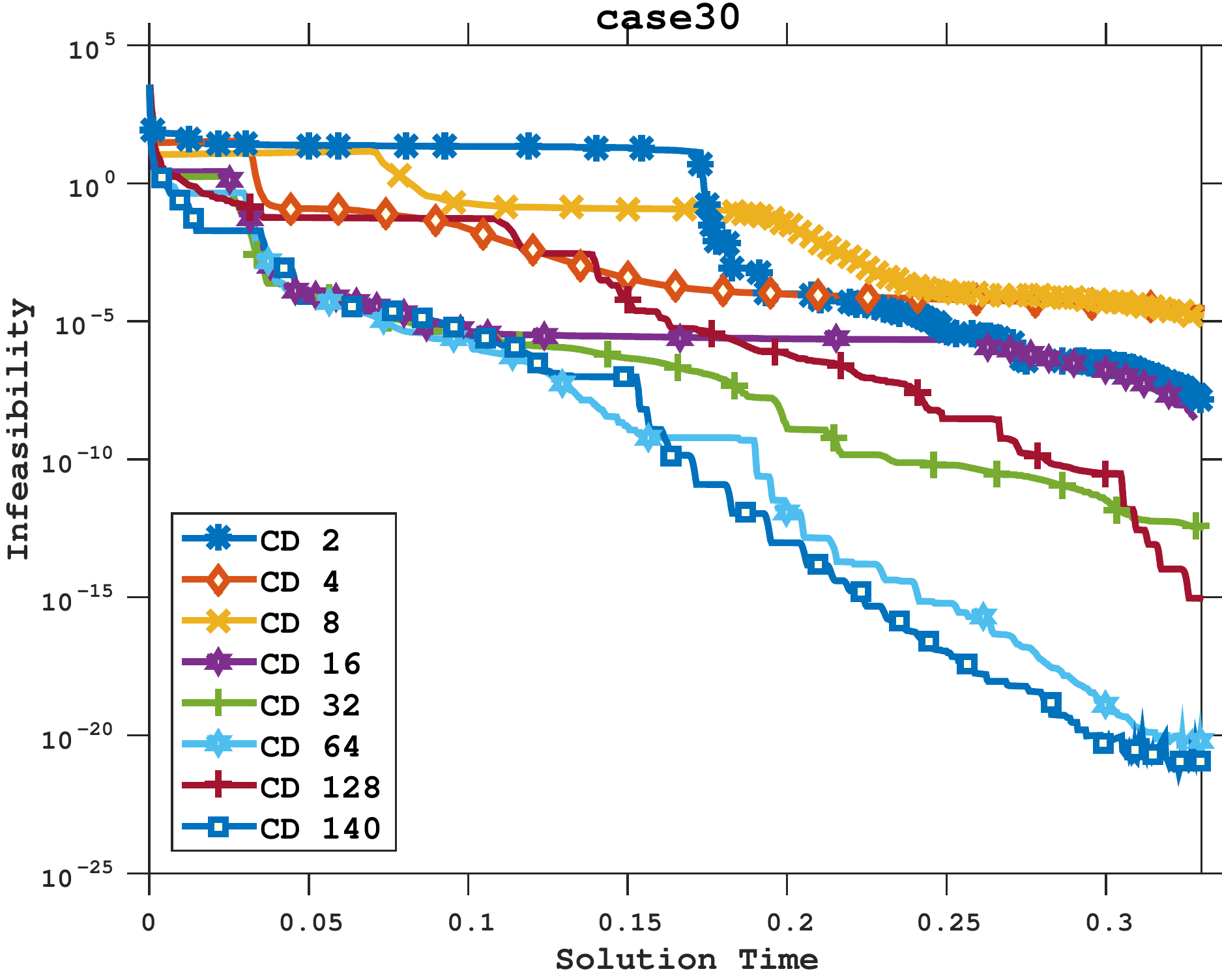}
  \includegraphics[width=0.32\textwidth]{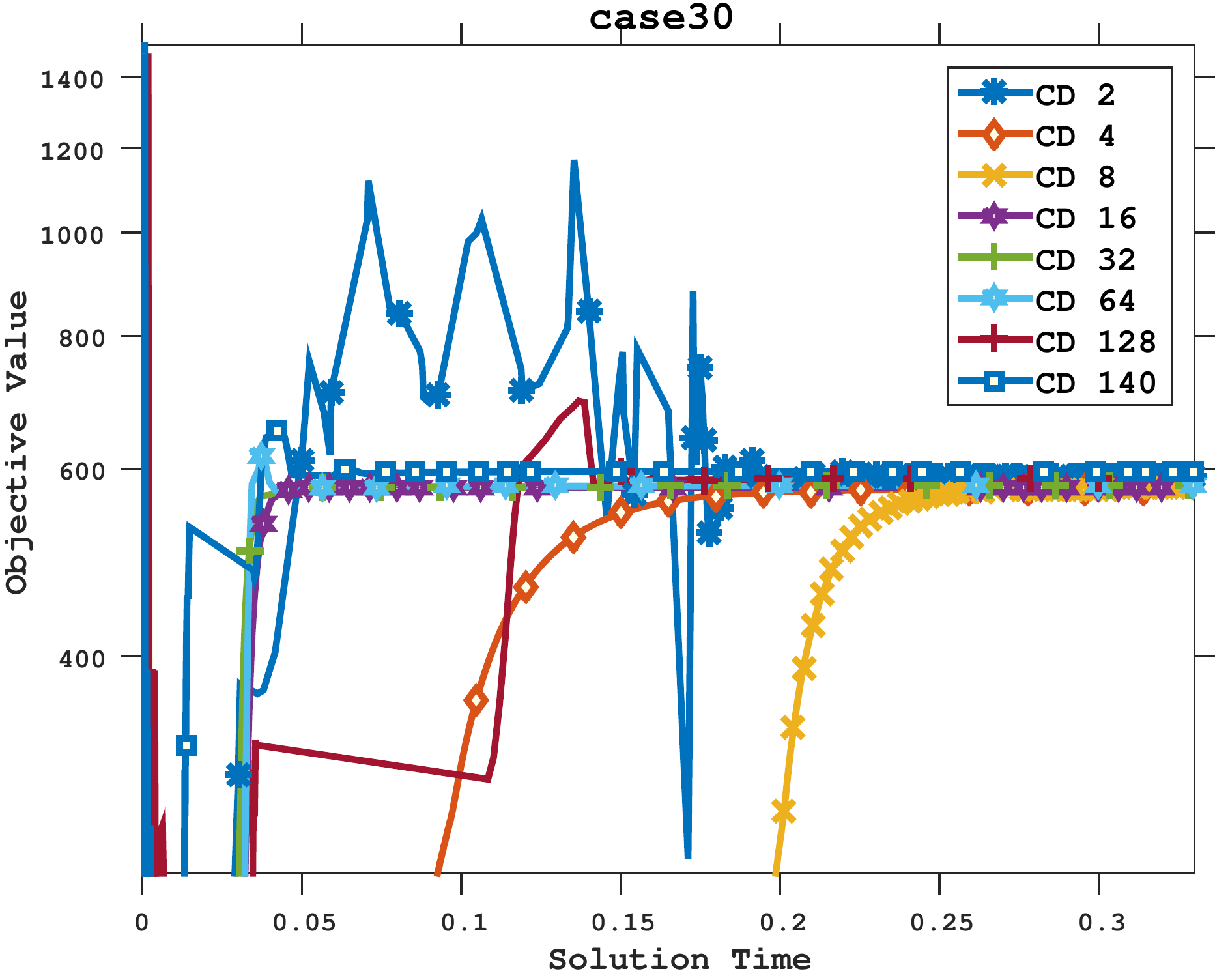}
  \includegraphics[width=0.32\textwidth]{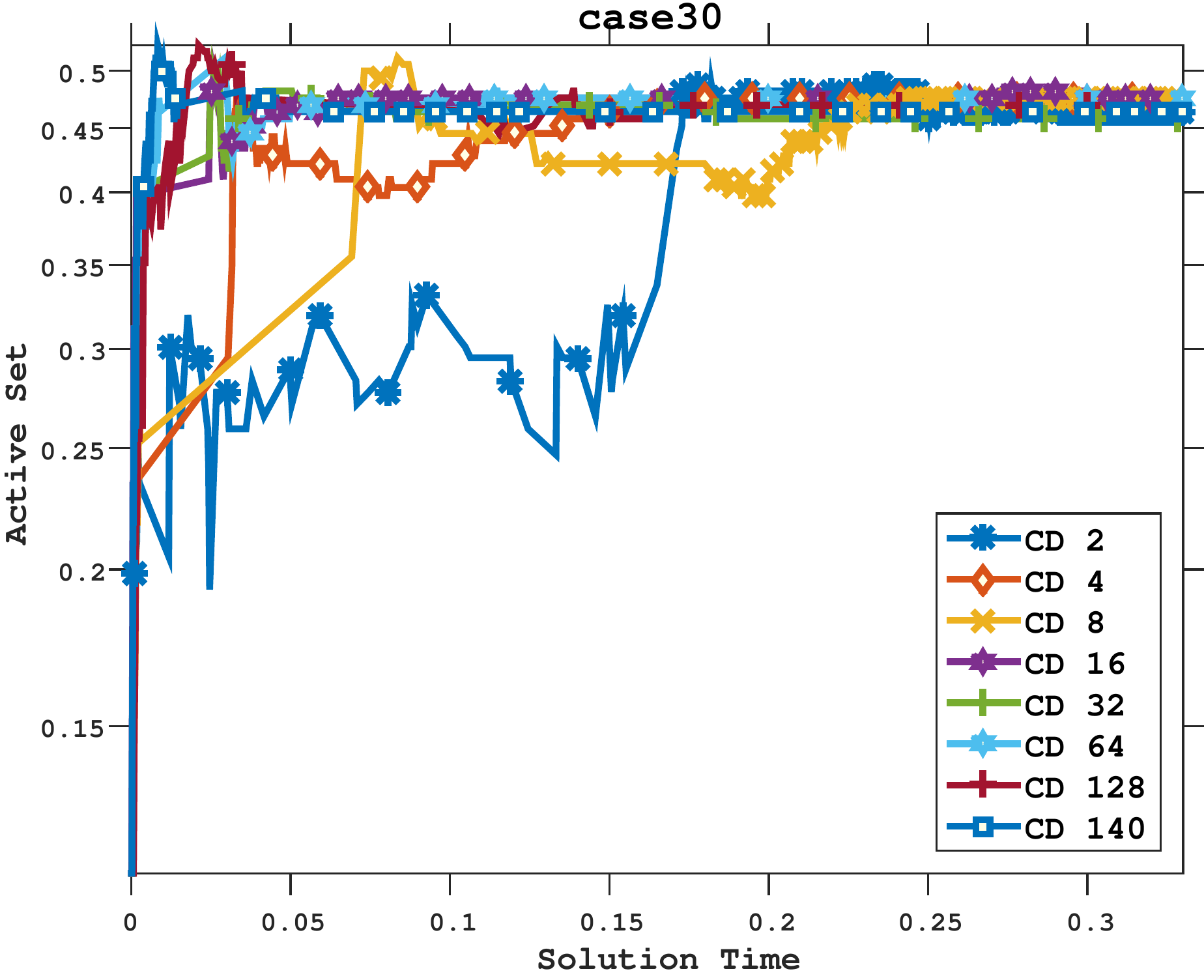}
  \\ $ $\\
  \includegraphics[width=0.32\textwidth]{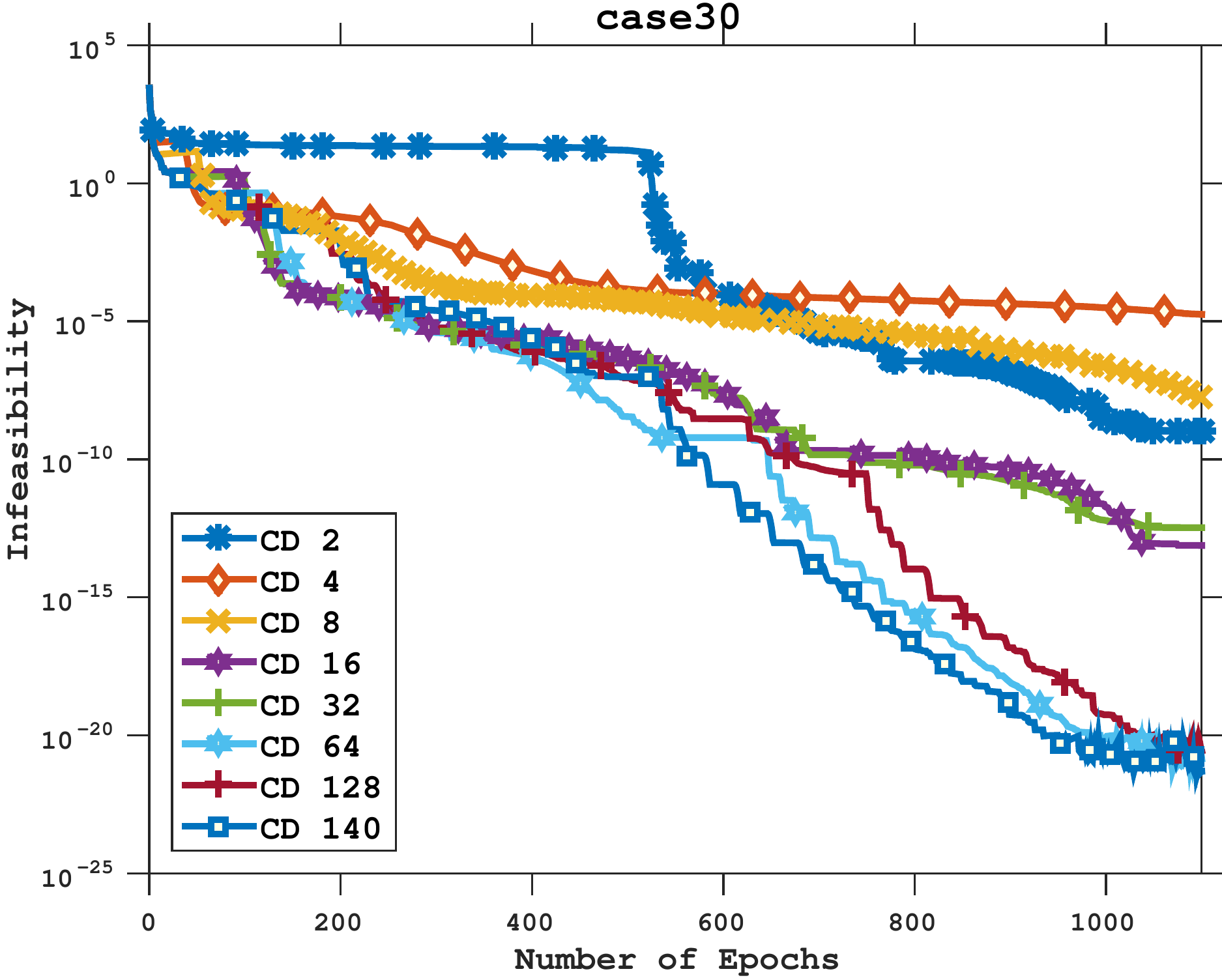}
  \includegraphics[width=0.32\textwidth]{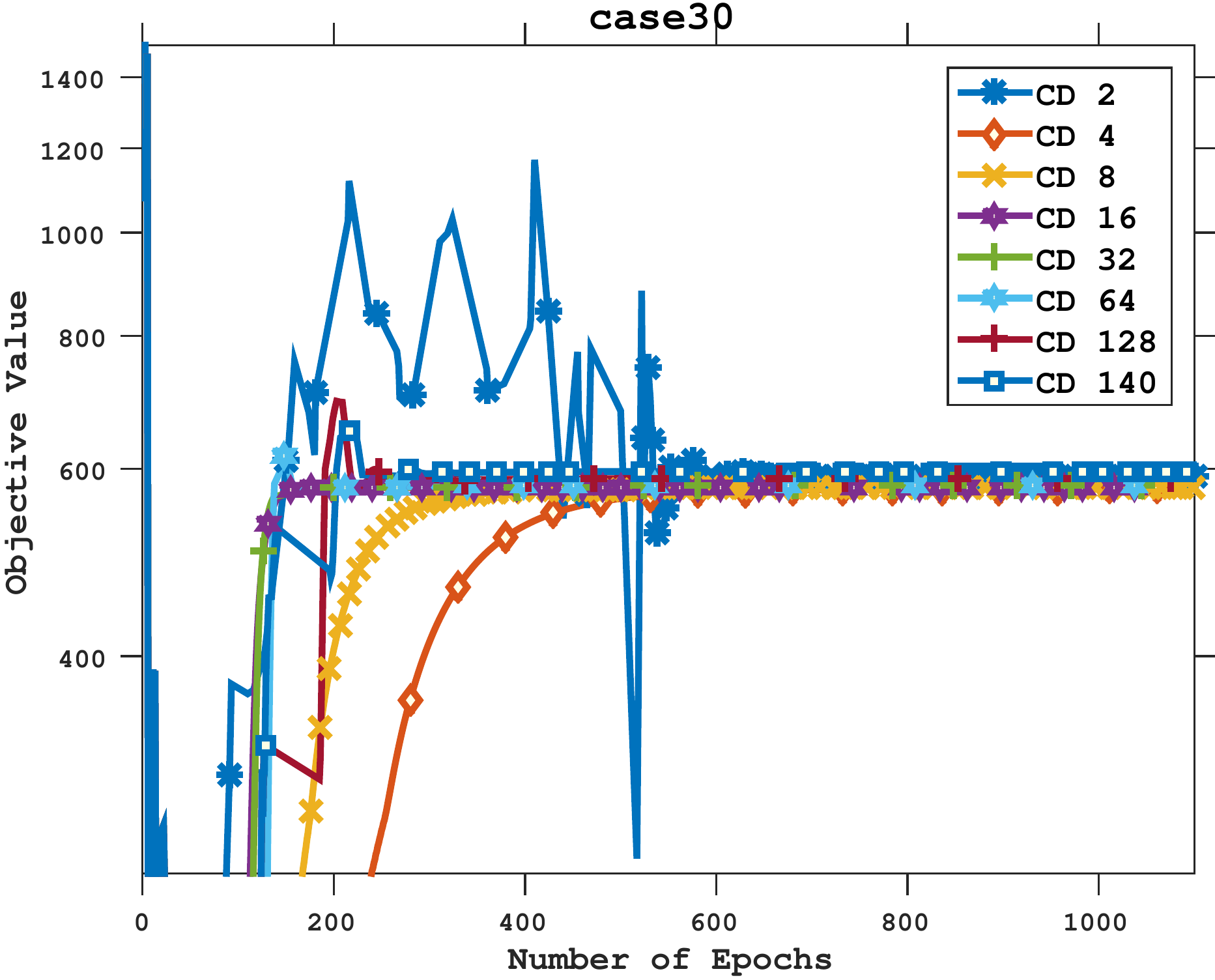}
  \includegraphics[width=0.32\textwidth]{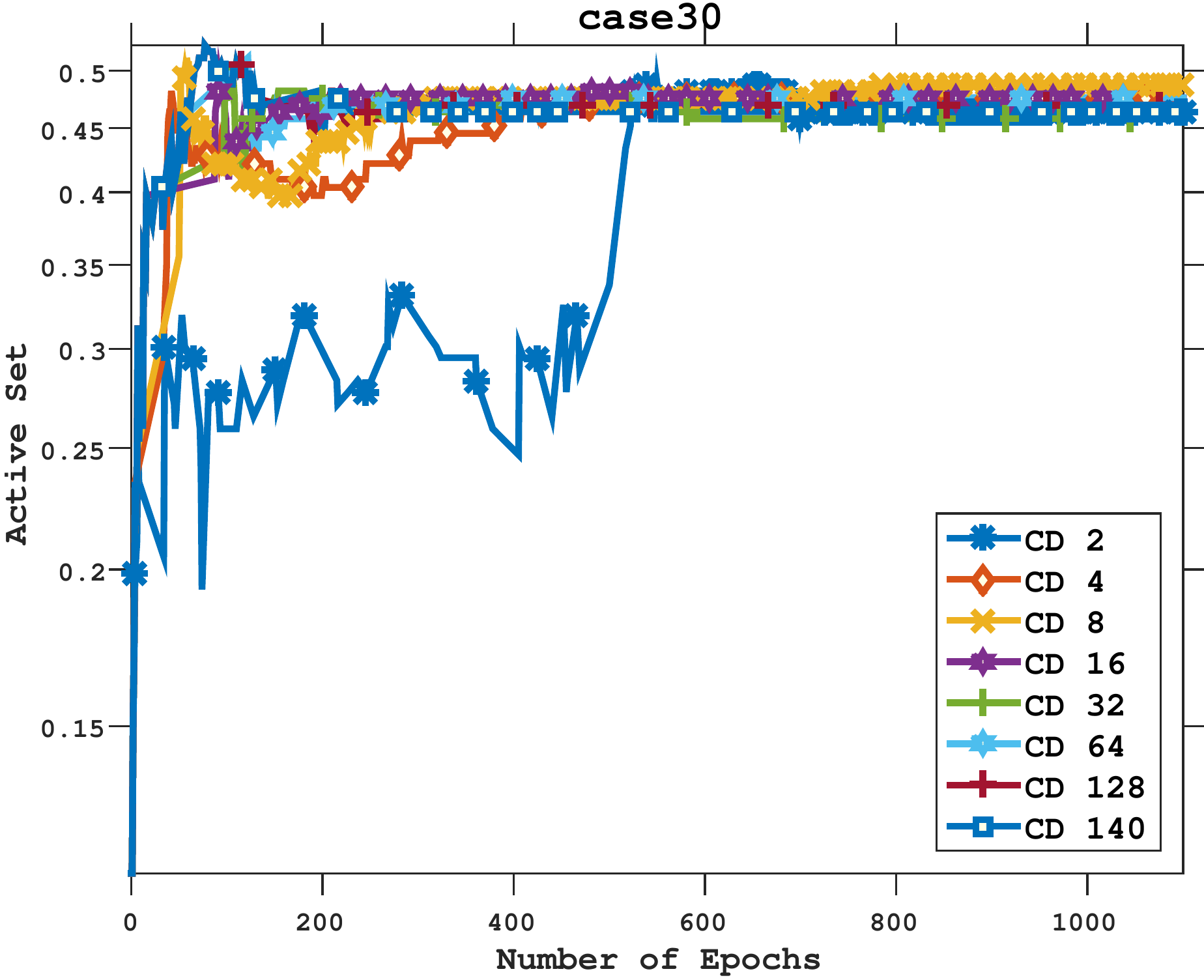}
  \caption{Results on IEEE 30-bus test case: infeasibility, objective function, and the proportion of constraints in the active set over time.}
  \label{fig:case30cdNewton}
\end{figure*}

\begin{figure*}[htbp]
  \center
  \includegraphics[width=0.32\textwidth]{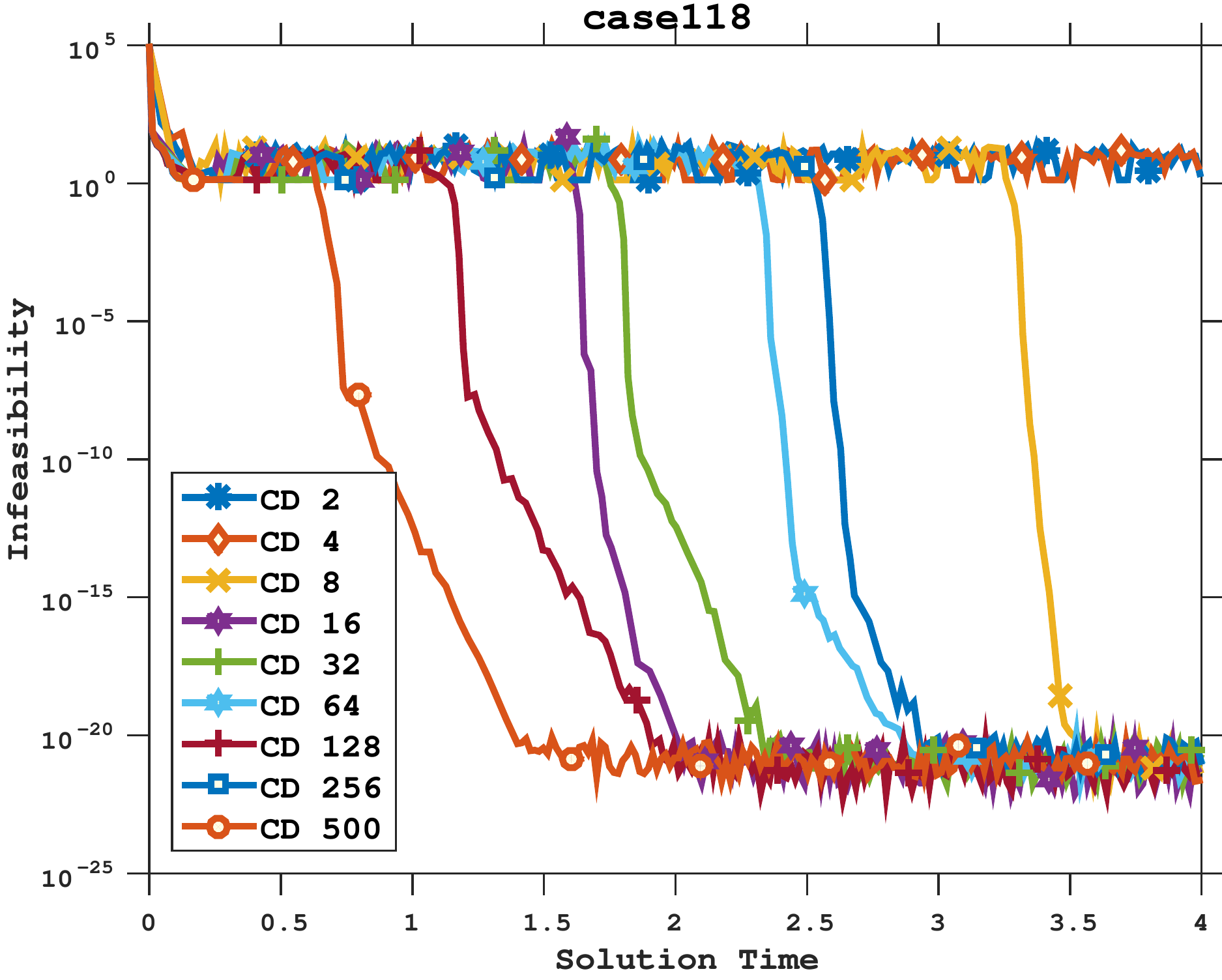}
  \includegraphics[width=0.32\textwidth]{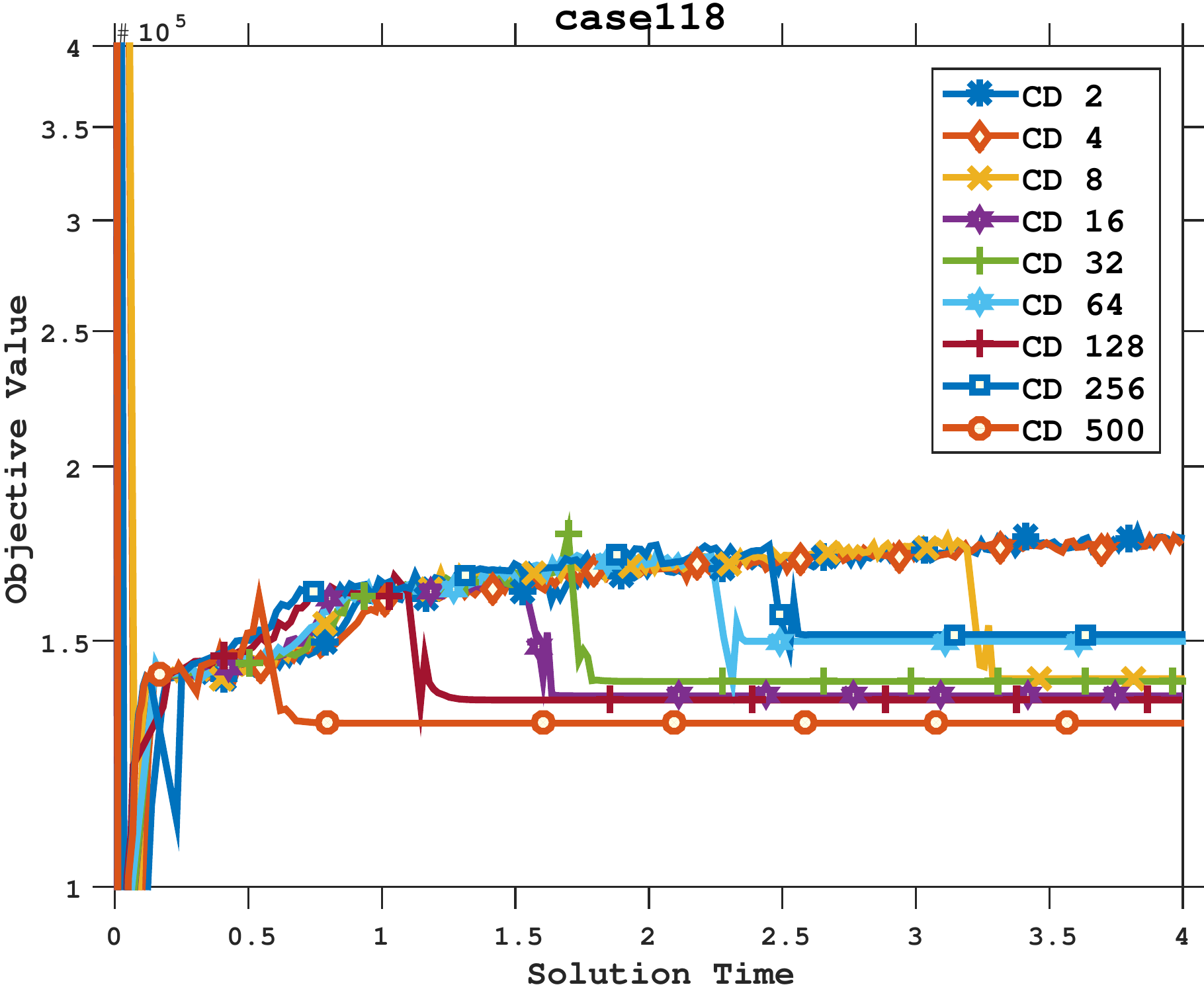}
  \includegraphics[width=0.32\textwidth]{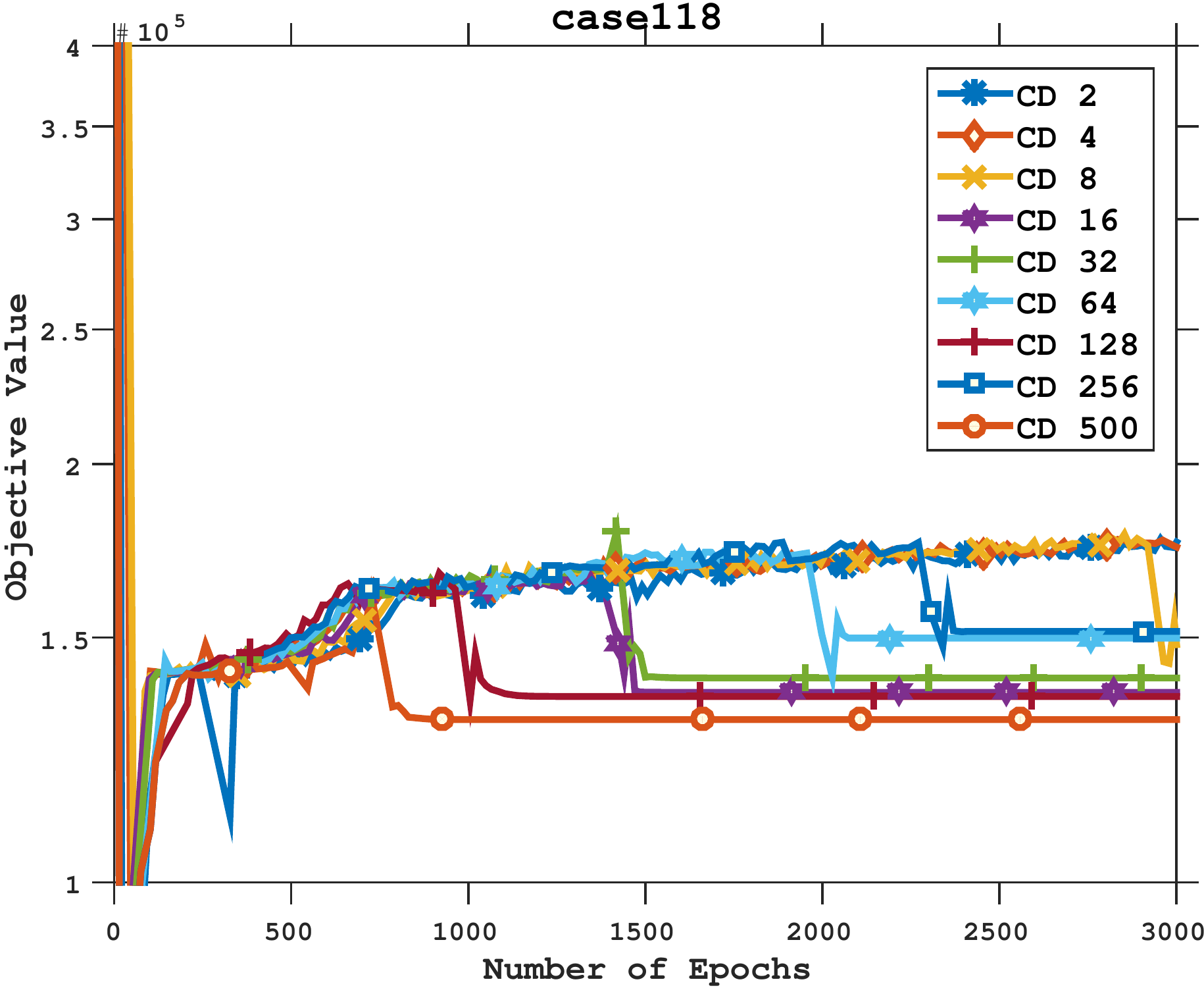}
  \caption{Results on IEEE 118-bus test case: infeasibility and objective function over time.}
  \label{fig:case118cdNewton}
\end{figure*}
\begin{figure*}[htbp]
  \center
  \includegraphics[width=0.32\textwidth]{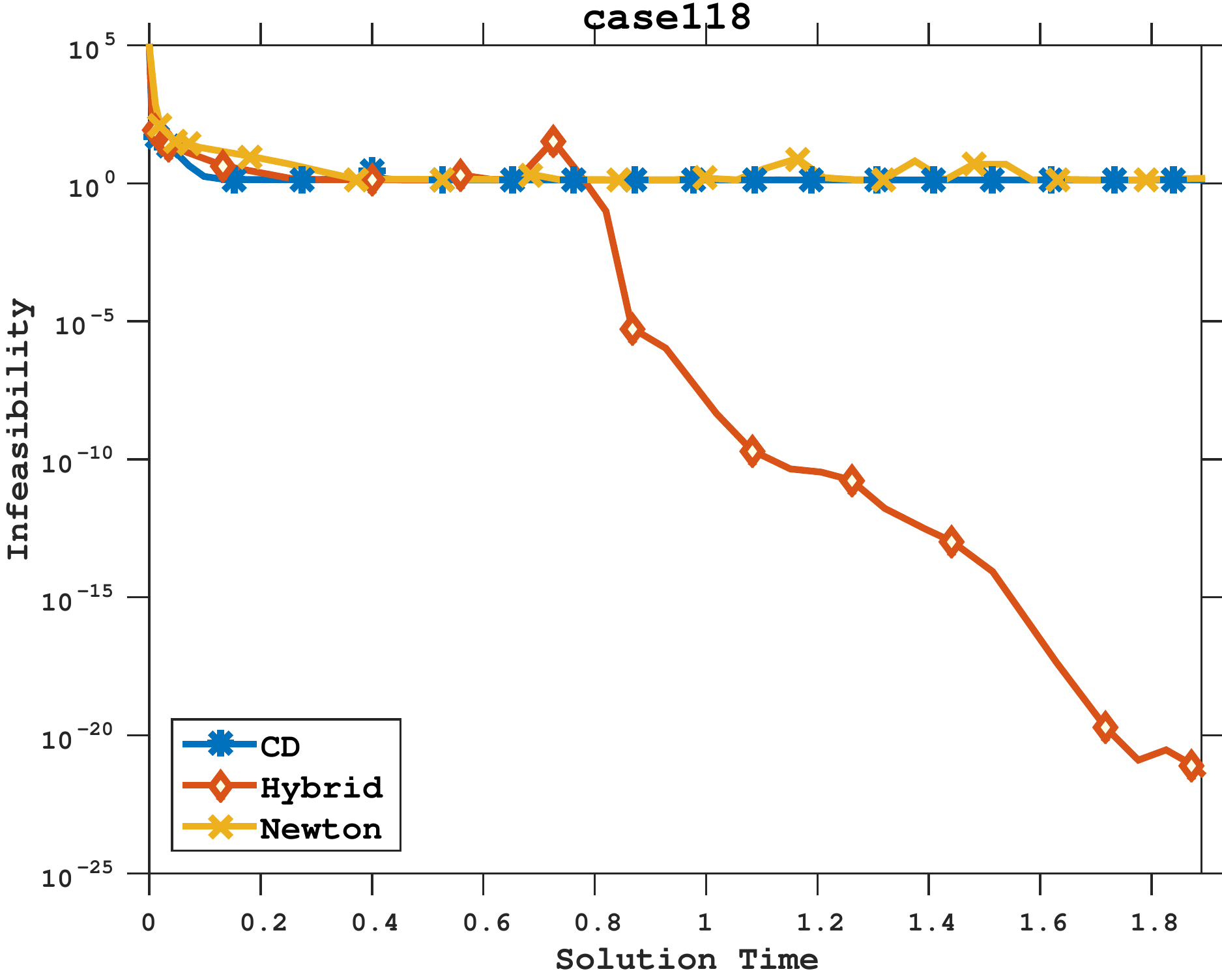}
  \includegraphics[width=0.32\textwidth]{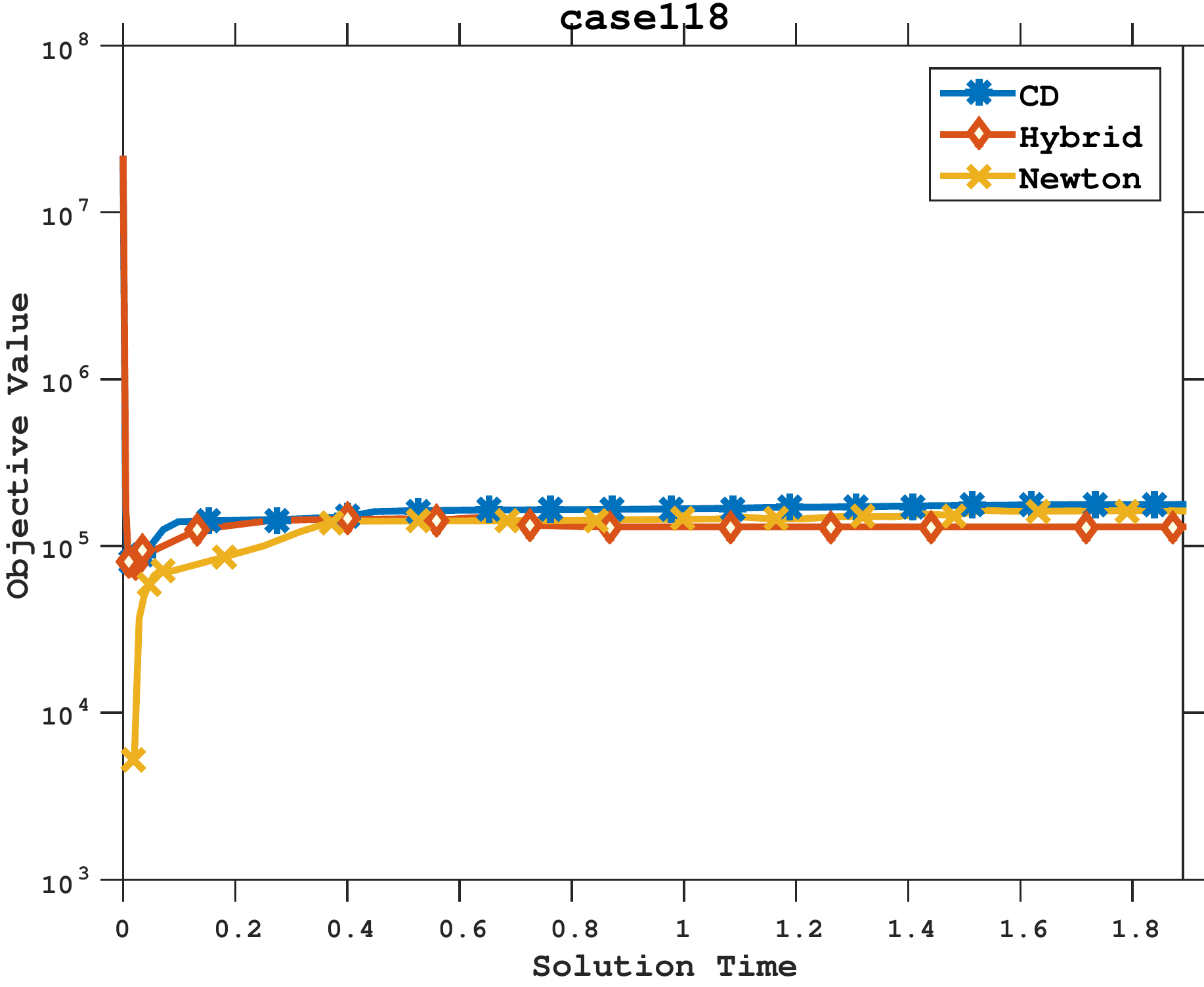}
  \\ $ $\\
  \includegraphics[width=0.32\textwidth]{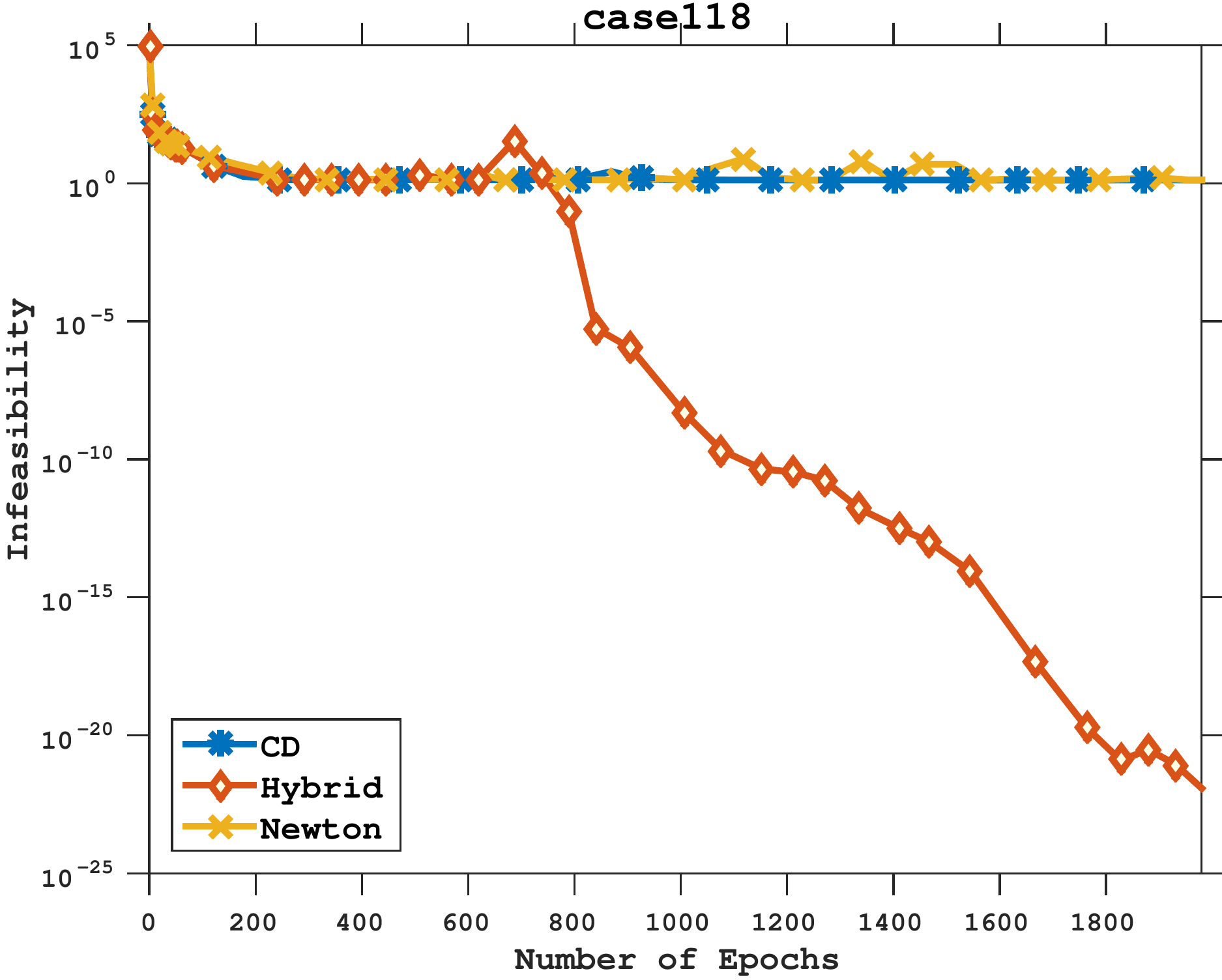}
  \includegraphics[width=0.32\textwidth]{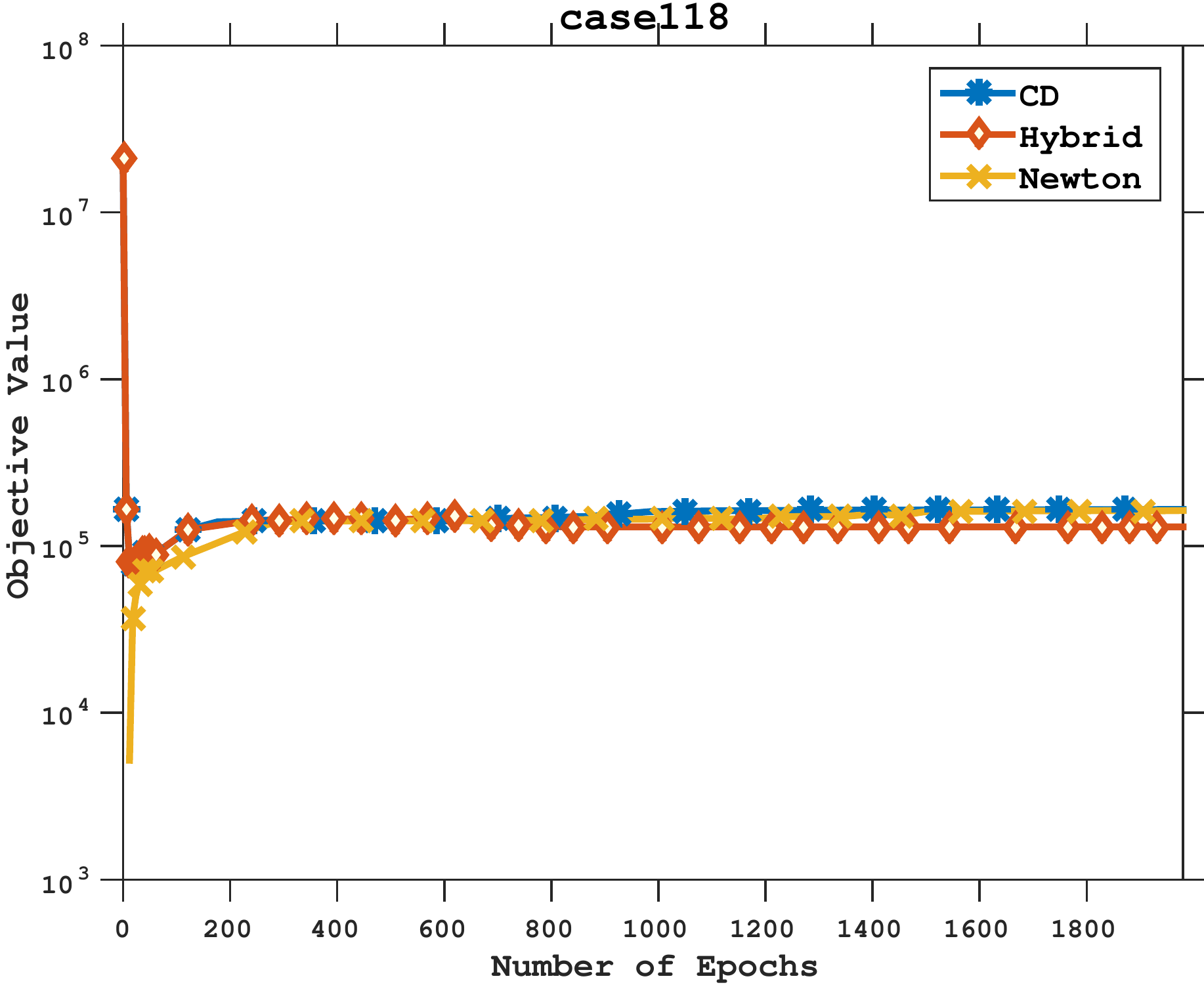}
  \caption{Results on IEEE 118-bus test case: infeasibility and objective function over time.}
  \label{fig:case118compare}
\end{figure*}
\begin{figure*}[htbp]
  \center
  \includegraphics[width=0.32\textwidth]{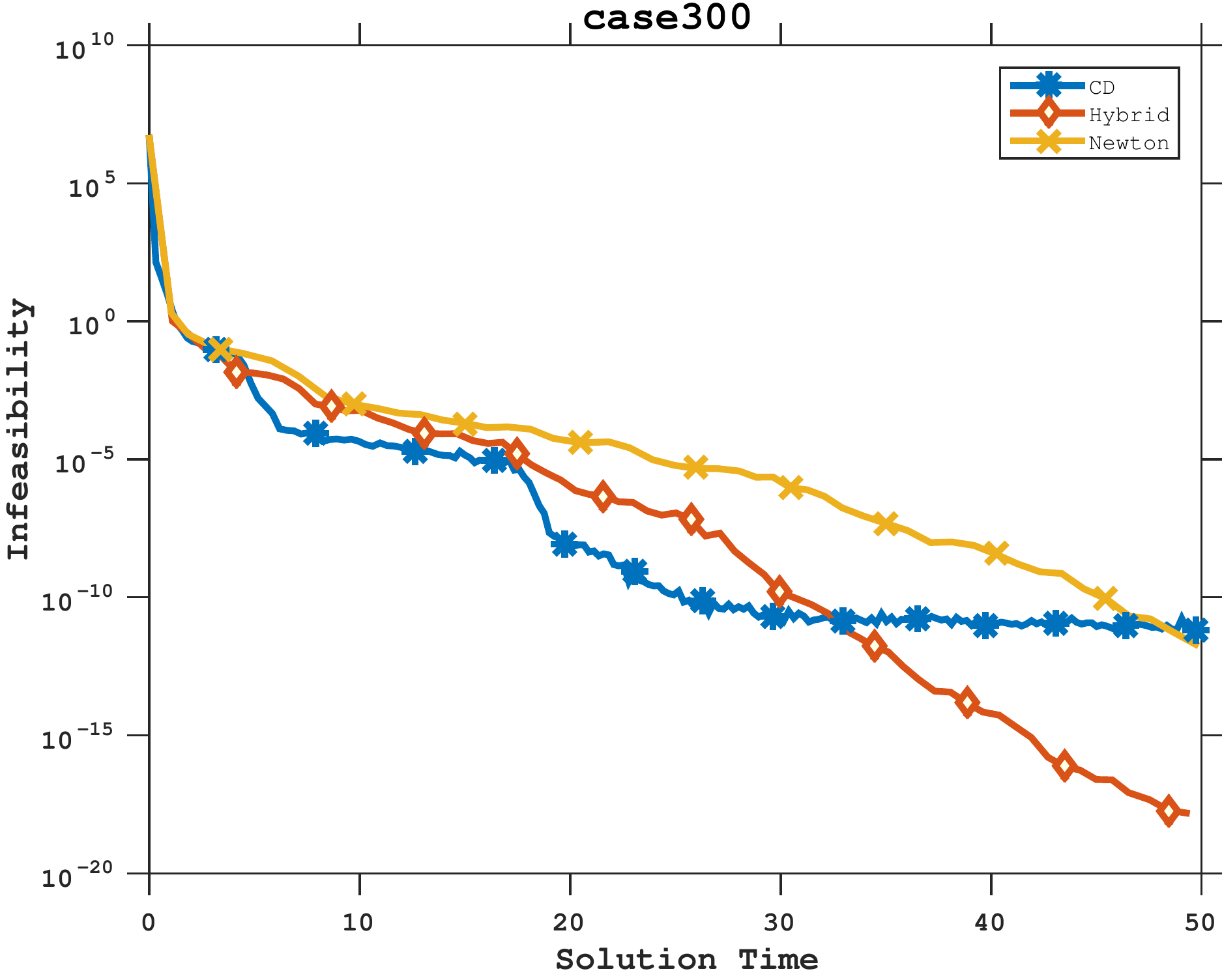}
  \includegraphics[width=0.32\textwidth]{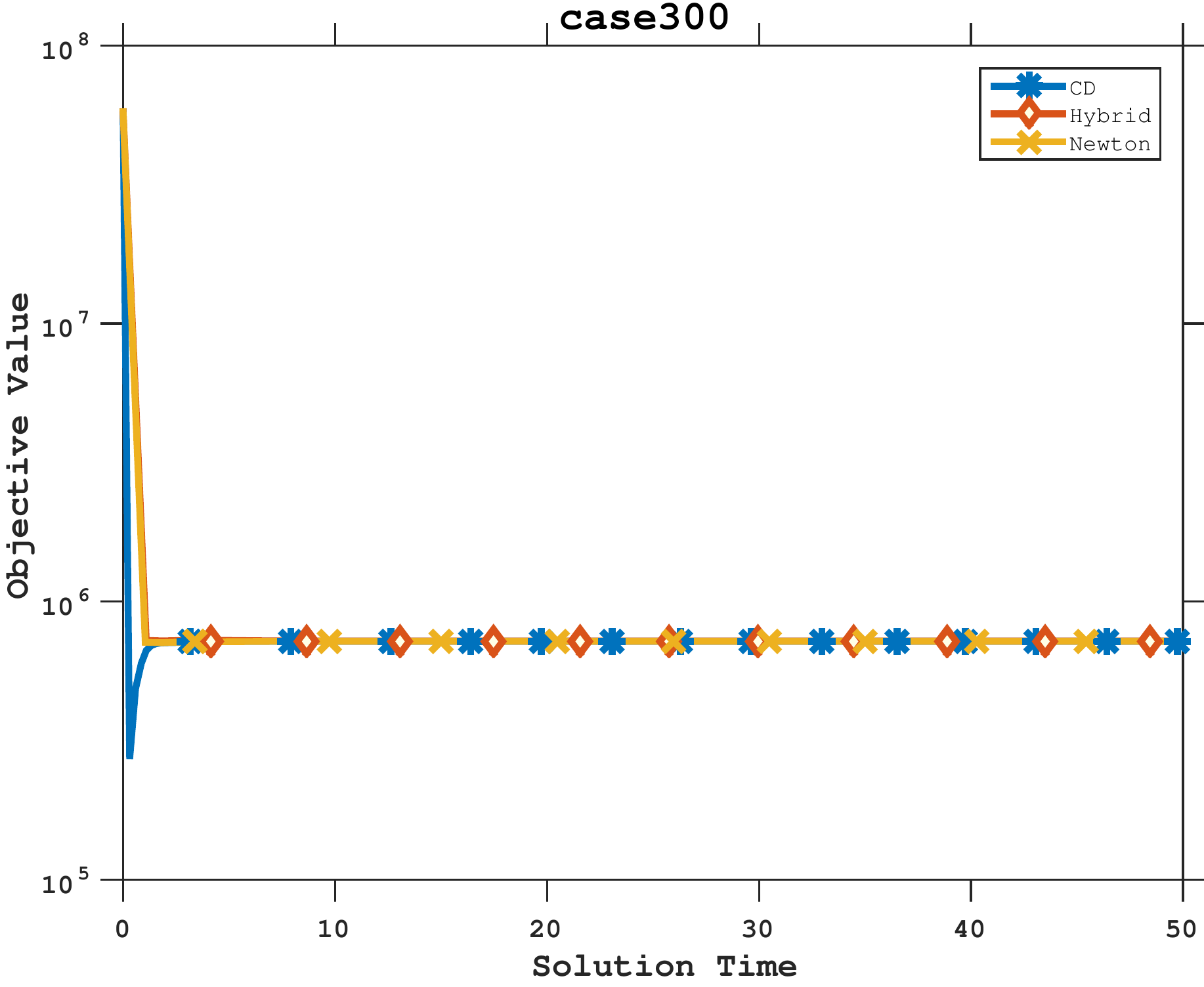}
  \includegraphics[width=0.32\textwidth]{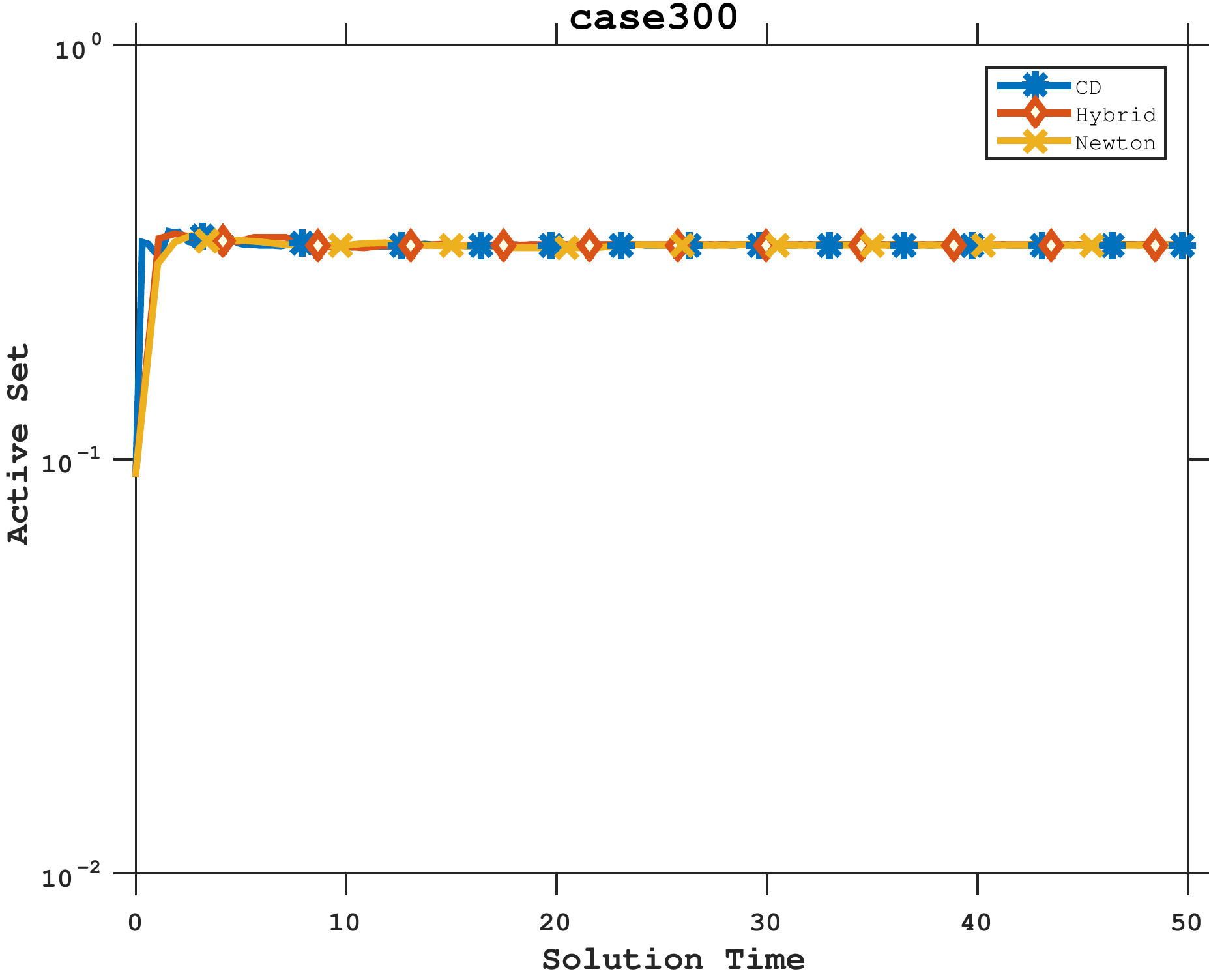}
  \caption{Results on IEEE 300-bus test case: infeasibility, objective function, and the proportion of constraints in the active set over time.}
  \label{fig:case300}
\end{figure*}
\begin{figure*}[htbp]
  \center
  \includegraphics[width=0.32\textwidth]{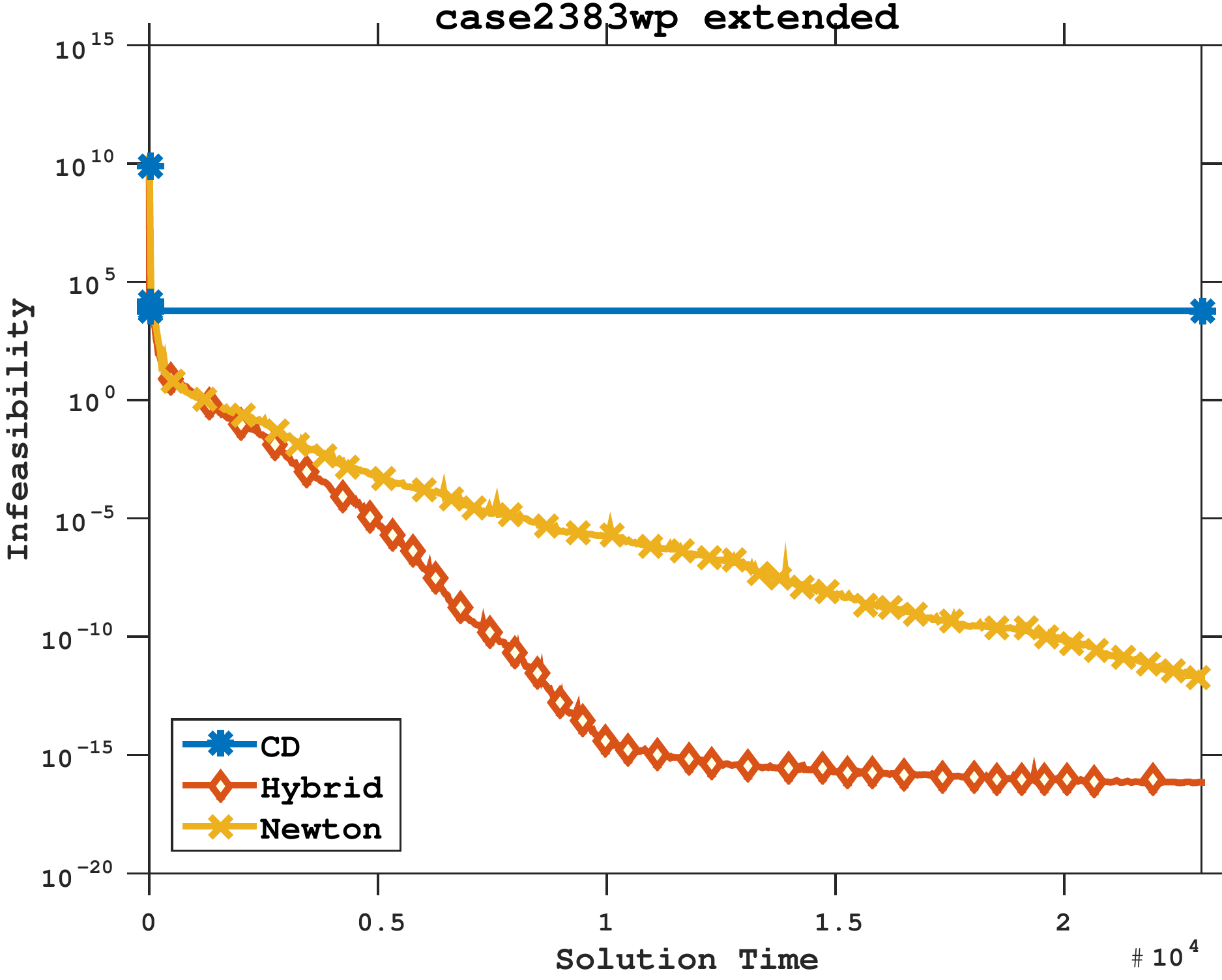}
  \includegraphics[width=0.32\textwidth]{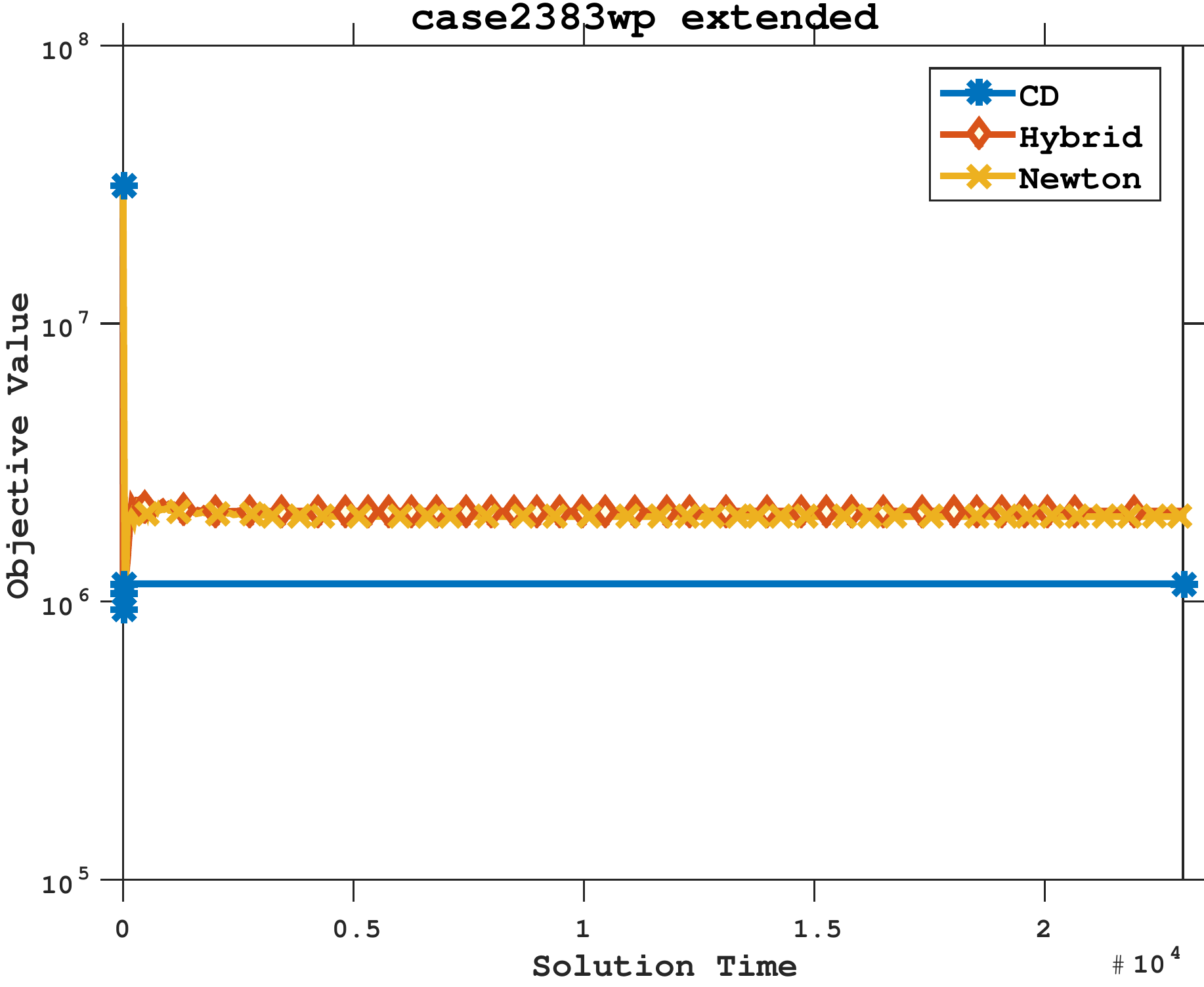}
  \\ $ $\\
  \includegraphics[width=0.32\textwidth]{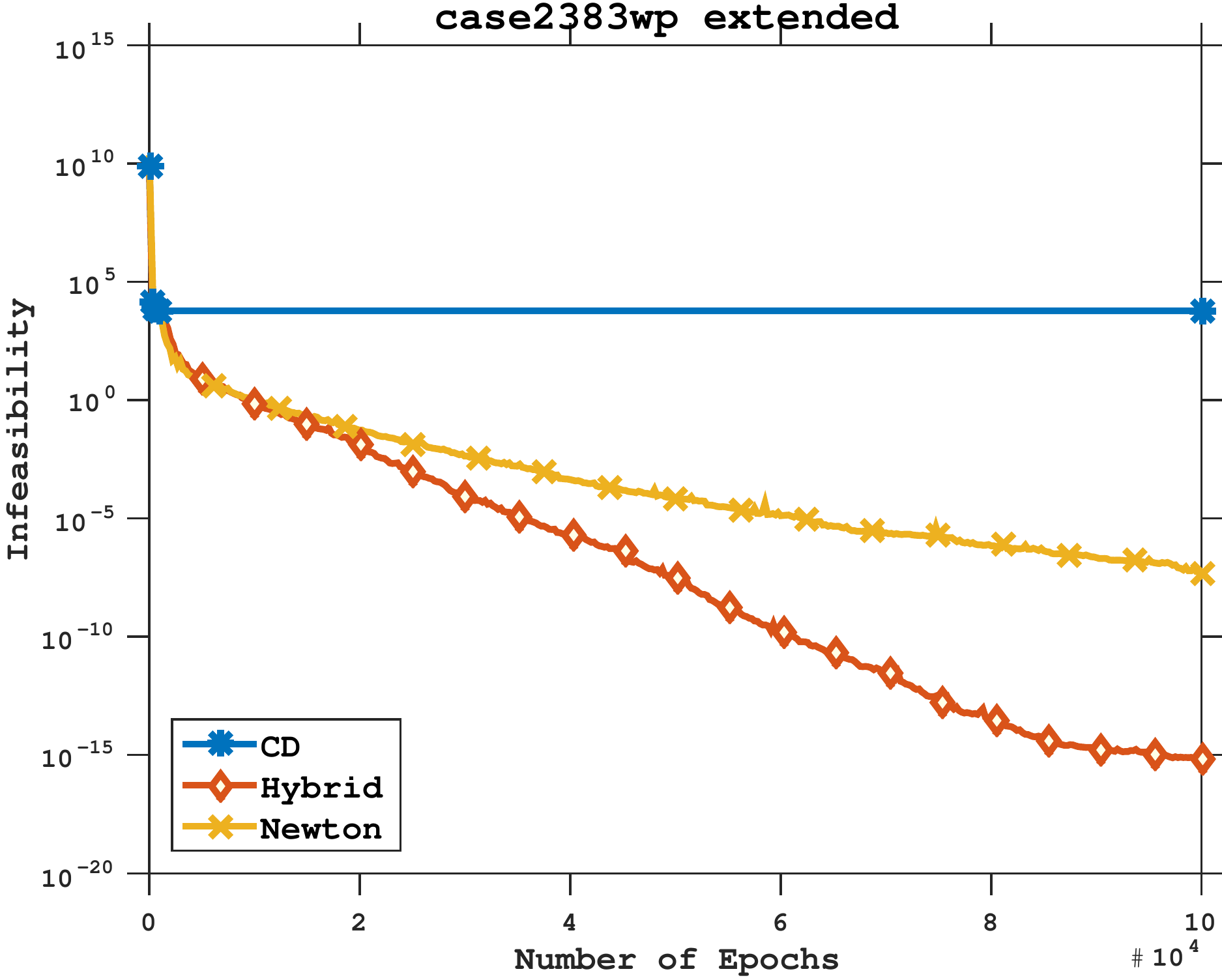}
  \includegraphics[width=0.32\textwidth]{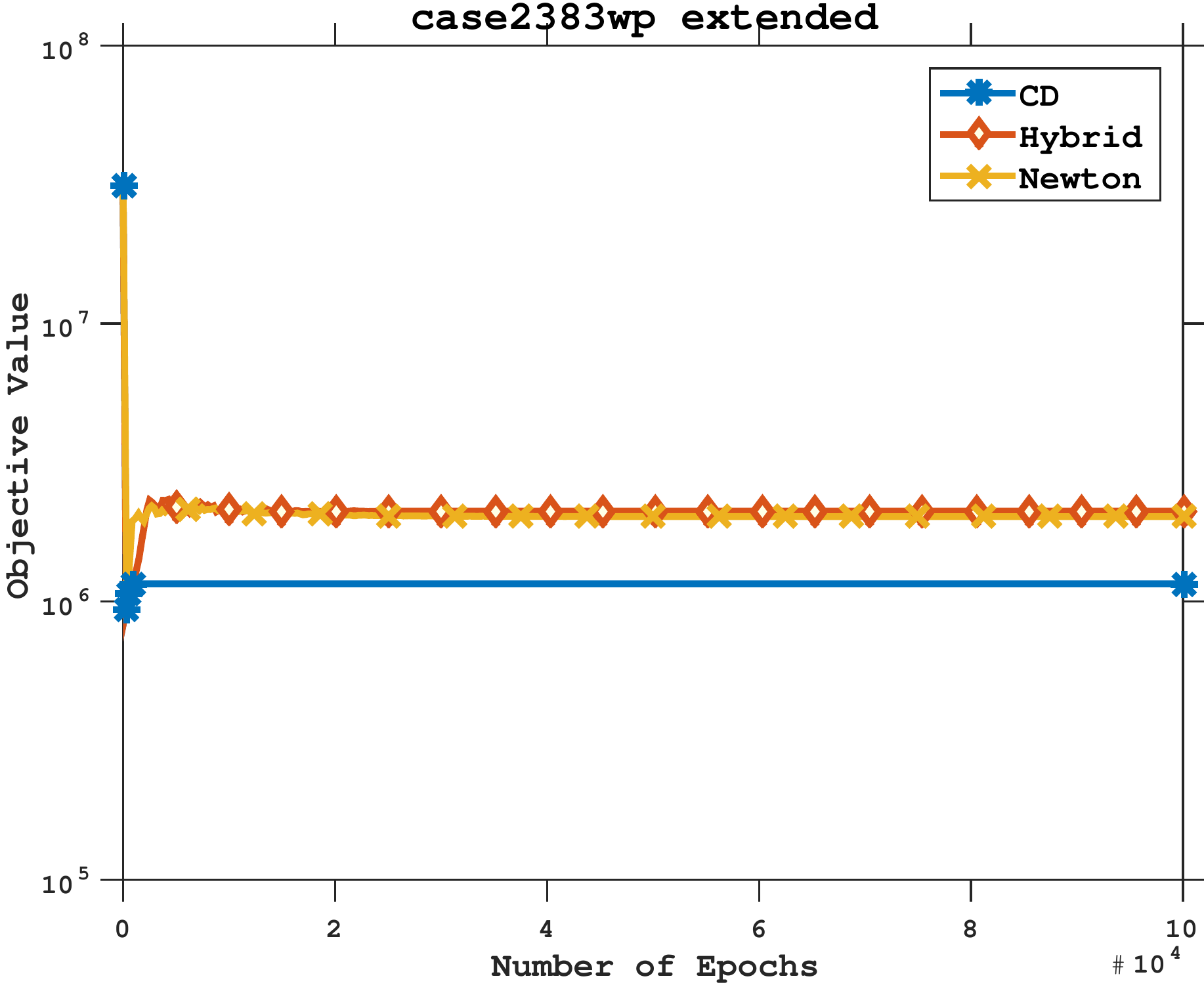}
  \caption{Results on 2383-bus test case modelling the transmission system of Poland (winter peak scenario): infeasibility and objective function over time.}
  \label{fig:case2383wp}
\end{figure*}

\section{Conclusions}
\label{sec:conclusions}

This paper proposes an approach for solving polynomial optimization problems to global optimality, by combining results from several different fields: moment relaxations, activity identification in non-linear programming, and point estimation (\ab{}) theory.
We propose an algorithm that combines a first-order method on the convex semidefinite relaxations that yields low-quality minimizers, with a second-order method on the polynomial optimization problem, that refines the minimizer quality.
The switch from the first to second-order method is based on a quantitative criterion which ensures quadratic convergence of Newton's method from the first iteration.
Finally, we illustrate our approach on instances of the optimal power-flow problem.

\section*{Acknowledgments}
We would like to acknowledge discussions with Alan C. Liddell and Jie Liu, both of whom contributed to the code for the experiments in the paper,
but chose not to be co-authors of the present submission.
This work has received funding from the European Union’s Horizon Europe research and innovation programme under grant agreement No. 101070568.
V.K. acknowledges support of the Czech Science Foundation (22-15524S).
J.M. acknowledges support of the Czech Science Foundation (23-07947S).

\bibliography{references}